\documentclass[12pt,twoside]{article}
\usepackage{amsrefs}
\usepackage{cite}
\usepackage{amsfonts,amssymb,amsmath,amsthm,mathrsfs}
\usepackage[T5,T1]{fontenc}
\usepackage[mathscr]{euscript}
\usepackage [latin1]{inputenc}
\usepackage{color}
\newcommand{\R}{\mathbb{R}}

\newtheorem*{theorem*}{Theorem}
\newtheorem{thm}{Theorem}[section]
\newtheorem{cor}[thm]{Corollary}
\newtheorem{lem}[thm]{Lemma}

\renewcommand{\le}{\leqslant}
\renewcommand{\leq}{\leqslant}
\renewcommand{\ge}{\geqslant}
\renewcommand{\geq}{\geqslant}

\usepackage{fancyhdr}
\begin{document}
\title{Global gradient estimates \\ for a general type of nonlinear parabolic equations}

\author{Cecilia Cavaterra${}^{(1,4)}$\and
Serena Dipierro${}^{(2)}$
\and
Zu Gao${}^{(3)}$
\and
Enrico Valdinoci${}^{(2)}$
}

\maketitle

{\scriptsize \begin{center} (1) -- Dipartimento di Matematica ``Federigo Enriques'',
Universit\`a degli Studi di Milano\\
Via Saldini 50, I-20133 Milano (Italy)\\
\end{center}
\scriptsize \begin{center} (2) -- Department of Mathematics and Statistics,
University of Western Australia\\ 35 Stirling Highway, WA6009 Crawley (Australia)\\
\end{center}
\scriptsize \begin{center} (3) --
Department of Mathematics, School of Science,
Wuhan University of Technology\\
122 Luoshi Road, 430070 Hubei, Wuhan (China) \end{center}
\scriptsize \begin{center} (4) -- Istituto di Matematica Applicata e Tecnologie Informatiche ``Enrico Magenes'', CNR\\
   Via Ferrata 1, 27100 Pavia (Italy)\\
\end{center}

\begin{center}
{\tt cecilia.cavaterra@unimi.it},
{\tt serena.dipierro@uwa.edu.au},
{\tt gaozu7@whut.edu.cn},
{\tt enrico.valdinoci@uwa.edu.au}
\end{center}
}
\bigskip\bigskip

\par
\noindent
\centerline{\today}
\begin{abstract}\noindent
We provide global gradient estimates for solutions to a general
type of nonlinear parabolic equations, possibly in a Riemannian geometry setting.

Our result is new in comparison with
the existing ones in the literature, in light of the validity of the estimates
in the global domain, and it detects several additional regularity effects due to
special parabolic data.

Moreover, our result comprises a large number of nonlinear sources
treated by a unified approach,
and it recovers many classical results as special cases.
\end{abstract}

\section{Introduction}

The goal of this paper is to consider a general
type of nonlinear parabolic equations, possibly in a Riemannian geometry setting,
and to provide new global gradient estimates.

The method that we use relies on the Maximum Principle,
as developed by Cheng and Yau
in~\cite{MR385749}
and Hamilton in~\cite{MR1230276},
and on suitable properties of the cut-off function
introduced by Li and Yau in \cite{6LY}, which are
also the key tool for the classical gradient estimates proved
by Souplet and Zhang
in~\cite{3SZ}.

Though
several gradient estimate results have been obtained in different cases
(see, e.g., \cite{MR2810695, 4CM, MR4238774,
5J, 7MZS, 8W, 9Z, 10Z, 11X, 13HM, 14DK, 14DKN, 15W, 16MZ, 17CZ}),
we provide here a general framework dealing,
at once, with various nonlinearities of interest (as a matter
of fact, a number of classical and recent
results can be re-obtained as special cases of our general
approach). Also, we will provide ``global'' (rather than ``local'')
estimates that take into account the parabolic boundary
behavior, thus improving the estimates when the data of the equation
are particularly favorable.\medskip

We point out that the pointwise gradient estimates
for parabolic equations have also a natural counterpart
for elliptic equations, see e.g.~\cite{2019arXiv190304569C,
MR454338, MR615561, MR803255, MR1296785, MR2680184, MR2812957, MR3231999},
and, in general, pointwise gradient
estimates based on Maximum Principles
are a classical, yet still very active, topic of investigation.\medskip

Now we introduce the mathematical framework in details.
Let~$\mathscr{M}$ be a Riemannian manifold of dimension $n\geq2$,
with Ricci curvature denoted by~$\mathrm{Ric}(\mathscr{M})$.
In this article, we will always suppose that the Ricci curvature of~$\mathscr{M}$
is bounded from below, namely,
\begin{equation}\label{ricci k}
\mathrm{Ric}(\mathscr{M})\geq-k,\end{equation}
for some $k\in\mathbb{R}$.

As customary, we also use the ``positive part'' notation
$$ k_{+}:=\max\{k,\,0\}.$$
The geodesic ball centered at~$x_0\in\mathscr{M}$
of radius~$R>0$ will be denoted by~$B(x_0,R)$.

Given~$x_0\in \mathscr{M}$,
$R>0$, $t_0\in\R$, and~$T>0$,
we consider a classical parabolic equation of the form
\begin{equation}\label{20}
u_t=\Delta u+S(x,t,u)\qquad{\mbox{ in }}Q_{R,T}.
\end{equation}
In this setting~$u=u(x,t)$, with~$x\in B(x_0,R)\subset\mathscr{M}$
and~$t\in[t_0-T,t_0]$,
and we have
used the classical notation
$$Q_{R,T}:=B(x_0,R)\times[t_0-T,t_0].$$
We will always suppose that
\begin{equation}\label{BOUND} u(x,t)\in(0,M]\qquad{\mbox{ for all }}(x,t)\in Q_{R,T},\end{equation}
for some~$M>0$.

Also, in~\eqref{20}
we denote by~$S$
a nonlinear source for the equation, that we
suppose to be~$ C^1$ in~$x$ and~$u$, and
continuous in~$t$.

In this setting, we define
\begin{equation}\label{GAMMA}
\gamma:=\sup_{(x,t,u)\in Q_{R,T}\times(0,M] }
\frac{\big|\nabla S(x,t,u)\big|}{u},
\end{equation}
where $\nabla$ stands for the
gradient with respect the components of the space variable~$x$.

Following~\cite{3SZ}, it is also convenient to consider the auxiliary function
\begin{equation}\label{def v}v(x,t):=\ln \frac{u(x,t)}{M}.\end{equation}
Furthermore, given~$k$ as in~\eqref{ricci k}, we set
\begin{equation}\label{LA:23} \mu:=\sup_{ (x,t)\in Q_{R,T}}
\left(k+
\partial_u S(x,t,u)-\frac{S(x,t,u)}{u}+\frac{S(x,t,u)}{u(1-v)}
\right)_+.\end{equation}
Our main goal is to establish global gradient bounds for
solutions of~\eqref{20}. Since these bounds may degenerate near
the parabolic boundary (e.g., if the initial or boundary data are
not regular enough), we exploit suitable cut-off functions.
Specifically, given~$\delta\in(0,T)$ and $\rho\in(0,R)$,
we consider the functions
\begin{equation}\label{161}\begin{split}&
{\mathscr{B}}_1(x,t):=\chi_{B(x_0,R-\rho)}(x)\chi_{[t_0-T, t_0-T+\delta)}(t),\\&
{\mathscr{B}}_2(x,t):=\chi_{B(x_0,R)\backslash B(x_0,R-\rho)}(x)\chi_{[t_0-T+\delta,t_0]}(t),\\&
{\mathscr{B}}_3(x,t):=\chi_{B(x_0,R)\backslash B(x_0,R-\rho)}(x)\chi_{[t_0-T,t_0-T+\delta)}(t)\\{\mbox{and }}\qquad&
{\mathscr{I}}(x,t):=\chi_{B(x_0,R-\rho)}(x)\chi_{[t_0-T+\delta,t_0]}(t).
\end{split}\end{equation}
We point out that the functions~${\mathscr{B}}_1$, ${\mathscr{B}}_2$
and~${\mathscr{B}}_3$ are localized in a neighborhood of the parabolic
boundary (namely, ${\mathscr{B}}_1$ near the time-boundary
but in the interior of the space-boundary,
${\mathscr{B}}_2$ near the space-boundary
but in the interior of the time-boundary,
and~${\mathscr{B}}_3$ near the space- and time-boundary).
Conversely, the function~${\mathscr{I}}$ is supported well-inside the
domain~$Q_{R,T}$, and
$$ {\mathscr{B}}_1+{\mathscr{B}}_2+{\mathscr{B}}_3+{\mathscr{I}}=
\chi_{Q_{R,T}},$$
hence the supports of these auxiliary functions can be seen as
a partition of the domain under consideration.

Furthermore, recalling the notation in~\eqref{GAMMA}
 and~\eqref{LA:23},
we define
\begin{equation}\label{TUTTITE-0}
\begin{split}
&{\mathscr{C}}:=
\gamma^{1/3}+\sqrt\mu,\\
&{\mathscr{T}}:=\frac{1}{\sqrt\delta},\\
{\mbox{and }}\quad&{\mathscr{S}}:=
\frac{1}{\rho}+\frac{1}{\sqrt{\rho(R-\rho)}}+
\frac{\sqrt[4]{k_+}}{\sqrt{\rho}}.
\end{split}
\end{equation}
We notice that~${\mathscr{C}}$, ${\mathscr{T}}$
and~${\mathscr{S}}$ are constants, depending on the nonlinearity,
the geometry of the manifold and the parameters~$R$, $\rho$ and~$\delta$.

Moreover, we set
\begin{equation}\label{sigtau}\begin{split}&
\tau_u:=\sup_{x\in B(x_0,R)}\frac{|\nabla u|}{u(1-v)}(x,t_0-T),
\\{\mbox{and }}\qquad&
\sigma_u:=\sup_{{x\in\partial B(x_0,R)}\atop{t\in[t_0-T,t_0]}}\frac{|\nabla u|}{u(1-v)}(x,t).
\end{split}\end{equation}
We remark that~$\tau_u$ and~$\sigma_u$ are known-objects,
once we know the parabolic boundary data of~$u$
(in particular, both~$\tau_u$ and~$\sigma_u$
are controlled by the supremum of~$\frac{|\nabla u|}{u(1-v)}$
over the parabolic boundary of~$Q_{R,T}$).

The terms in~\eqref{TUTTITE-0}
are the building blocks of our chief estimate, since they comprise
the different pointwise behavior of the solution, in different regions
of the space-time domain. More specifically:
\begin{itemize}
\item the term ${\mathscr{C}}$ is a ``common'' term in all the domain,
produced by the nonlinearity~$S$ and by the curvature of the ambient manifold,
\item the term ${\mathscr{T}}$ is a localization term due to a cut-off
function in the time variable,
\item the term ${\mathscr{S}}$ is a localization term due to a cut-off
function in the space variable.
\end{itemize}
Our strategy would then be to choose,
in our chief estimate, the ``best option''
between the boundary datum and the universal smoothing effect
produced by the heat equation. To this end, given a constant~$C>0$ (that will be taken conveniently large in the following Theorem~\ref{TH2})
we define
\begin{equation}\label{162}\begin{split}
\beta_1\,&:=
\tau_u +\min\left\{\sigma_u,\,C{{\mathscr{S}}}\right\},\\
\beta_2\,&:=
\sigma_u +\min\left\{\tau_u,\,C{{\mathscr{T}}}\right\},\\
\beta_3\,&:=\sigma_u +\tau_u,\\ {\mbox{and }}\qquad
\iota\,&:=
\min\left\{\sigma_u+\tau_u,\,
\sigma_u+C{{\mathscr{T}} },\,
\tau_u+C{{\mathscr{S}}} ,\,
C({{\mathscr{T}}}+{{\mathscr{S}}})\right\}
.\end{split}\end{equation}
We point out that the quantities in~\eqref{162}
are constants.
In this framework, the term~$\beta_1$
will take care of the region of the domain
in the interior of the space variable and near the parabolic boundary
in the time variable: such a term takes into account the initial datum
in time and, thanks to the spatially interior smoothing effect,
takes the ``best possible choice'' between the spatial boundary datum
and the estimate produced by a spatial cut-off.

Similarly,
the term~$\beta_2$ in~\eqref{162}
will take care of the region of the domain
in the interior of the time variable and near the spatial boundary: such a term takes into account the boundary datum
in space and, thanks to the smoothing effect for positive times,
takes the ``best possible choice'' between the initial datum
and the estimate produced by a cut-off in the time variable.

The term~$\beta_3$ in~\eqref{162} deals with the case
of proximity to the boundary for both the space and time variables,
and clearly reflects the influence of the data on the whole of the
parabolic boundary.

Finally, the term~$\iota$ in~\eqref{162} considers the case
of interior points, both in space and time: in this case,
one can take the ``best possible choice'' between the data along the
parabolic boundary and the universal smoothing effect
of the heat equation in the interior of the domain.

In view of these considerations,
the coupling between the auxiliary functions and the coefficients is
encoded by the function
\begin{equation}\label{DEFdiZ}
\mathscr{Z}:=
\beta_1\,{\mathscr{B}}_1
+\beta_2\,{\mathscr{B}}_2
+\beta_3\,{\mathscr{B}}_3
+\iota\,{\mathscr{I}}.
\end{equation}
With this notation, the main estimate of this paper goes as follows:

\begin{thm}\label{TH2}
Suppose that $u$ is a solution of equation~\eqref{20}
satisfying~\eqref{BOUND}.

Then,
there exists~$C>0$, only depending on~$n$, such that,
for any~$\delta\in(0,T)$ and $\rho\in(0,R)$, we have
that
\begin{equation}\label{B}
\aligned
&~~~~\frac{|\nabla u(x,t)|}{u(x,t)} \leq \Big(C
\mathscr{C}+
\mathscr{Z}(x,t)\Big)\,
\left(1+\ln \frac{M}{u(x,t)}\right)\quad{\mbox{ for all }}
(x,t)\in Q_{R,T}.
\endaligned
\end{equation}
\end{thm}

We observe that the estimate in~\eqref{B} is dimensionally
coherent: indeed, taking the time variable
to have the same measure units
of the square of the space variable, and~$S$ to have the same units of~$u$
over the square of the space variables, we obtain that
all the terms in~\eqref{B} have the units of the inverse of the space
(for this, we recall that the Ricci curvature scales like the inverse
of the square of the space variable, see for instance
Example 8.5.5 on page~418 and Theorem 8.5.22 on page~427 in~\cite{MR2327126}).\medskip

We also point out that gradient estimates
as the one in~\eqref{B} can be also considered as extension
of classical gradient estimates inspired by the Bernstein technique,
see~\cite{MR0114050, MR1351007, MR3450752, MR0402274, MR3254790, MR3450752, MR4021092, SIRALORO, BERNOI}.

With respect to this feature, we stress that the quantities~$\tau_u$ and~$\sigma_u$ introduced
in~\eqref{sigtau} are not just natural and convenient objects motivated by the methods
relying on the Maximum Principle, but they play an interesting role
in {\em improving the known estimates for cases
in which the parabolic data are particularly nice}. More specifically,
the classical estimates have the striking feature of being ``universal'',
and thus independent on the data: this property is certainly advantageous
when dealing with poor boundary data, since these classical results still ensure suitable gradient bounds in the interior,
but they become somehow ``suboptimal'' when the boundary data are extremely good (these are precisely the pros and cons
of an estimate holding true regardless of the specific boundary conditions!).
Instead, the estimates that we provide in this paper are flexible enough, on the one hand,
to recover several classical results of universal type and, on the other hand, to provide enhanced
estimates when the data are extraordinarily nice (this general concept will be quantified,
for instance, at the end of Section~\ref{ECBD-90SINSCIUSTU}).\medskip

The rest of this paper is organized as follows.
We provide the proof of
Theorem~\ref{TH2} in Section~\ref{KAseecdt:2}.
Then, in Section~\ref{KAseecdt:3}, we give some
specific applications of our general result,
also showing how it comprises and improves
some classical and recent results from the literature.

\section{Proof of Theorem~\ref{TH2}}\label{KAseecdt:2}
\setcounter{equation}{0}
\vskip2mm
\noindent

The proof of Theorem~\ref{TH2} relies on a general estimate,
given in the forthcoming Lemma~\ref{PI}.
With that, one will obtain the desired claim in Theorem~\ref{TH2}
by considering different regions, according to the cut-off functions,
and exploiting the Maximum Principle in the interior.

To this aim, we set
\begin{equation}\label{DE w}
w:=\frac{|\nabla v|^2}{(1-v)^{2}},\end{equation}
and we have:

\begin{lem}\label{PI} Let~$u$ be as in Theorem~\ref{TH2},
$v$ be as in~\eqref{def v} and~$w$ be as in~\eqref{DE w}.

Then, in~$Q_{R,T}$, it holds
\begin{equation*}
\frac{
\Delta w-w_t}2\,\ge\,
(1-v)\,w^2+
\frac{v\,\langle\nabla w,\nabla v\rangle}{1-v}-\frac{\gamma\,|\nabla v|}{(1-v)^2}
-\mu w,
\end{equation*}
where~$\gamma$ and~$\mu$ are as in~\eqref{GAMMA}
and~\eqref{LA:23}.
\end{lem}

\begin{proof}
Recalling~\eqref{BOUND} and~\eqref{def v}, we see that
\begin{equation}\label{Jna:0}v\leq0,\end{equation} and
\begin{equation}\label{Jna:1}
v_t=\frac{u_t}{u},\qquad
\nabla v=\frac{\nabla u}{u}\qquad{\mbox{and}}\qquad
\Delta v=\frac{u\Delta u-|\nabla u|^2}{u^2}.
\end{equation}
As a consequence, from~\eqref{20}, we have

\begin{equation}\label{2e2}\begin{split}
v_t\,&=\frac{\Delta u+S(x,t,u)}{u}\\&=\frac{u\Delta u-|\nabla u|^2}{u^2}+
\frac{|\nabla u|^2}{u^2}
+\frac{S(x,t,u)}{u}
\\&=
\Delta v+|\nabla v|^2+\frac{S(x,t,u)}{u}.
\end{split}\end{equation}

Now we observe that
\begin{equation}\label{nu:1}
\nabla w
=\frac{\nabla|\nabla v|^2}{(1-v)^{2}}+2\frac{|\nabla v|^2\nabla v}{(1-v)^{3}}.
\end{equation}
Moreover, we have that
\begin{equation}\label{nu:2}
{\rm div}\,\left(\frac{\nabla|\nabla v|^2}{(1-v)^{2}}\right)=
\frac{\Delta|\nabla v|^2}{(1-v)^{2}}+
\frac{2\langle\nabla|\nabla v|^2,\nabla v\rangle}{(1-v)^{3}}.
\end{equation}
In addition,
\begin{equation*}
{\rm div}\,\left(\frac{|\nabla v|^2\nabla v}{(1-v)^{3}}\right)=
\frac{\langle\nabla |\nabla v|^2,\nabla v\rangle}{(1-v)^{3}}
+
\frac{|\nabla v|^2\Delta v}{(1-v)^{3}}+
\frac{3|\nabla v|^4}{(1-v)^{4}}.
\end{equation*}
{F}rom this, \eqref{nu:1} and~\eqref{nu:2}, we deduce that
\begin{equation}\label{nu:3}
\begin{split}&
\Delta w\\=\;&
\frac{\Delta|\nabla v|^2}{(1-v)^{2}}+
\frac{2\langle\nabla|\nabla v|^2,\nabla v\rangle}{(1-v)^{3}}
+
\frac{2\langle\nabla |\nabla v|^2,\nabla v\rangle}{(1-v)^{3}}
+
\frac{2|\nabla v|^2\Delta v}{(1-v)^{3}}+
\frac{6|\nabla v|^4}{(1-v)^{4}}
\\=\;&\frac{\Delta|\nabla v|^2}{(1-v)^2}+\frac{4\langle\nabla|\nabla v|^2,\nabla v\rangle}{(1-v)^3}
+\frac{2|\nabla v|^2\Delta v}{(1-v)^3}+\frac{6|\nabla v|^4}{(1-v)^4}.
\end{split}\end{equation}
Moreover, using~\eqref{2e2}, we find that
\begin{equation}\nonumber
\aligned
w_t&=\frac{2\langle\nabla v,\nabla v_t\rangle}{(1-v)^2}+\frac{2|\nabla v|^2v_t}{(1-v)^3}
\\&=\frac{2\langle\nabla v,\nabla\Delta v\rangle}{(1-v)^2}
+\frac{2\langle\nabla v,\nabla|\nabla v|^2\rangle}{(1-v)^2}
+\frac{2\left\langle\nabla v,\nabla \left(
\frac{S(x,t,u)}{u}\right)\right\rangle}{(1-v)^2}\\&~~~~
+\frac{2|\nabla v|^2\Delta v}{(1-v)^3}
+\frac{2|\nabla v|^4}{(1-v)^3}+\frac{2|\nabla v|^2\frac{S(x,t,u)}{u}}{(1-v)^3}
.\endaligned
\end{equation}
This and~\eqref{nu:3}, after the cancellation of one term, give that
\begin{equation}\label{nu:9}
\aligned
\Delta w-w_t&=\frac{\Delta|\nabla v|^2}{(1-v)^2}-
\frac{2\langle\nabla v,\nabla\Delta v\rangle}{(1-v)^2}
-\frac{2\langle\nabla v,\nabla|\nabla v|^2\rangle}{(1-v)^2}
-\frac{2\left\langle\nabla v,\nabla\left(
\frac{S(x,t,u)}{u}\right)\right\rangle}{(1-v)^2}
\\&~~~~+\frac{4\langle\nabla|\nabla v|^2,\nabla v\rangle}{(1-v)^3}
-\frac{2|\nabla v|^4}{(1-v)^3}-\frac{2 |\nabla v|^2\frac{S(x,t,u)}{u}}{(1-v)^3}+\frac{6|\nabla v|^4}{(1-v)^4}.
\endaligned
\end{equation}

Now we recall the Bochner's formula, according to which
$$ \Delta {\bigg (}{\frac {|\nabla v|^{2}}{2}}{\bigg )}=\langle \nabla \Delta v,\nabla v\rangle +|D^2 v|^{2}+{\mbox{Ric}}(\nabla v,\nabla v).$$
This and~\eqref{ricci k} entail that
\begin{eqnarray*} \Delta|\nabla v|^{2}-2\langle \nabla \Delta v,\nabla v\rangle &=&
2|D^2 v|^{2}+2{\mbox{Ric}}(\nabla v,\nabla v)\\&\ge&
2|D^2 v|^{2}-2k\,|\nabla v|^2
.\end{eqnarray*}
Plugging this information in~\eqref{nu:9}, we conclude that
\begin{equation}\label{nu:19}
\aligned
\Delta w-w_t&\ge \frac{2|D^2 v|^{2}-2k\,|\nabla v|^2}{(1-v)^2}
-\frac{2\langle\nabla v,\nabla|\nabla v|^2\rangle}{(1-v)^2}
-\frac{2\left\langle\nabla v,\nabla\left(
\frac{S(x,t,u)}{u}\right)\right\rangle}{(1-v)^2}
\\&~~~~+\frac{4\langle\nabla|\nabla v|^2,\nabla v\rangle}{(1-v)^3}
-\frac{2|\nabla v|^4}{(1-v)^3}-\frac{2 |\nabla v|^2\frac{S(x,t,u)}{u}}{(1-v)^3}+\frac{6|\nabla v|^4}{(1-v)^4}.
\endaligned
\end{equation}
We also remark that
\begin{eqnarray*}0&\le&\left(\frac{\langle\nabla|\nabla v|^2,\nabla v\rangle}{2|\nabla v|^2}
+\frac{|\nabla v|^2}{1-v}\right)^2\\
&=&
\left(\frac{\langle\nabla|\nabla v|^2,\nabla v\rangle}{2|\nabla v|^2}\right)^2
+
\frac{\langle\nabla|\nabla v|^2,\nabla v\rangle}{1-v}
+\frac{|\nabla v|^4}{(1-v)^2}\\&=&
\left(\frac{\langle D^2\,v \nabla v,\nabla v\rangle}{|\nabla v|^2}\right)^2
+
\frac{\langle\nabla|\nabla v|^2,\nabla v\rangle}{1-v}
+\frac{|\nabla v|^4}{(1-v)^2}\\&\le&
|D^2v|^2+
\frac{\langle\nabla|\nabla v|^2,\nabla v\rangle}{1-v}
+\frac{|\nabla v|^4}{(1-v)^2}.
\end{eqnarray*}
{F}rom this and~\eqref{nu:19}, one finds that
\begin{equation}\label{nu:20}
\aligned
\Delta w-w_t&\ge -\frac{2k\,|\nabla v|^2}{(1-v)^2}
-\frac{2\langle\nabla v,\nabla|\nabla v|^2\rangle}{(1-v)^2}
-\frac{2\left\langle\nabla v,\nabla\left(
\frac{S(x,t,u)}{u}\right)\right\rangle}{(1-v)^2}
\\&~~~~+\frac{2\langle\nabla|\nabla v|^2,\nabla v\rangle}{(1-v)^3}
-\frac{2|\nabla v|^4}{(1-v)^3}-\frac{2 |\nabla v|^2\frac{S(x,t,u)}{u}}{(1-v)^3}+\frac{4|\nabla v|^4}{(1-v)^4}.
\endaligned
\end{equation}
Furthermore, in light of~\eqref{nu:1},
\begin{equation}\nonumber
\langle\nabla w,\nabla v\rangle=\frac{\langle\nabla v,\nabla|\nabla v|^2\rangle}{(1-v)^2}+\frac{2|\nabla v|^4}{(1-v)^3},
\end{equation}
and, as a result,
\begin{eqnarray*}
&&2\langle\nabla w,\nabla v\rangle=\frac{2\langle\nabla v,\nabla|\nabla v|^2\rangle}{(1-v)^2}+\frac{4|\nabla v|^4}{(1-v)^3}
\\{\mbox{and }}
&&\frac{2\langle\nabla w,\nabla v\rangle}{1-v}=\frac{2\langle\nabla v,\nabla|\nabla v|^2\rangle}{(1-v)^3}+\frac{4|\nabla v|^4}{(1-v)^4}.
\end{eqnarray*}
These identities, combined with~\eqref{nu:20}, yield that
\begin{equation*}
\aligned
\Delta w-w_t&\ge -\frac{2k\,|\nabla v|^2}{(1-v)^2}
-\frac{2\left\langle\nabla v,\nabla\left(
\frac{S(x,t,u)}{u}\right)\right\rangle}{(1-v)^2}
\\&~~~~
+\frac{2|\nabla v|^4}{(1-v)^3}-\frac{2 |\nabla v|^2\frac{S(x,t,u)}{u}}{(1-v)^3}-
2\langle\nabla w,\nabla v\rangle+
\frac{2\langle\nabla w,\nabla v\rangle}{1-v}.
\endaligned
\end{equation*}
We rewrite this formula as
\begin{equation*}
\aligned\frac{
\Delta w-w_t}2&\ge -\frac{k\,|\nabla v|^2}{(1-v)^2}
-\frac{\left\langle\nabla v,\nabla\left(
\frac{S(x,t,u)}{u}\right)\right\rangle}{(1-v)^2}
\\&~~~~
+\frac{|\nabla v|^4}{(1-v)^3}-\frac{ |\nabla v|^2\frac{S(x,t,u)}{u}}{(1-v)^3}+
\frac{v\,\langle\nabla w,\nabla v\rangle}{1-v}.
\endaligned
\end{equation*}
As a result, recalling~\eqref{DE w}, we conclude that
\begin{equation}\label{nu:21}
\frac{
\Delta w-w_t}2\ge -kw
-\frac{\left\langle\nabla v,\nabla\left(
\frac{S(x,t,u)}{u}\right)\right\rangle}{(1-v)^2}
+(1-v)w^2-\frac{ w\,S(x,t,u)}{(1-v)u}+
\frac{v\,\langle\nabla w,\nabla v\rangle}{1-v}.
\end{equation}
We also exploit~\eqref{Jna:1} to write that~$\nabla u=u\nabla v$, and accordingly
\begin{eqnarray*} \nabla\left(\frac{S(x,t,u)}{u}\right)&=&
\frac{\nabla S(x,t,u)}{u}+
\frac{\partial_u S(x,t,u)\nabla u}{u}
-\frac{S(x,t,u)\nabla u}{u^2}\\&=&
\frac{\nabla S(x,t,u)}{u}+\left(
\partial_u S(x,t,u)
-\frac{S(x,t,u)}{u}\right)\nabla v.
\end{eqnarray*}
Consequently,
\begin{eqnarray*}
\frac{\left\langle\nabla v,\nabla\frac{S(x,t,u)}{u}\right\rangle}{(1-v)^2}&=&
\frac{\langle\nabla S(x,t,u),\nabla v\rangle}{(1-v)^2u}+\left(
\partial_u S(x,t,u)
-\frac{S(x,t,u)}{u}\right)\frac{|\nabla v|^2}{(1-v)^2}\\&=&
\frac{\langle\nabla S(x,t,u),\nabla v\rangle}{(1-v)^2u}+\left(
\partial_u S(x,t,u)
-\frac{S(x,t,u)}{u}\right)\,w.
\end{eqnarray*}
This, \eqref{GAMMA}, \eqref{LA:23} and~\eqref{nu:21} lead to
\begin{equation*}\begin{split}
\frac{
\Delta w-w_t}2\,\ge\,&
(1-v)w^2+
\frac{v\,\langle\nabla w,\nabla v\rangle}{1-v}-\frac{\langle\nabla S(x,t,u),\nabla v\rangle}{(1-v)^2u}\\&\qquad-
\left(k+
\partial_u S(x,t,u)
-\frac{S(x,t,u)}{u}
+\frac{ S(x,t,u)}{(1-v)u}
\right)\,w\\
\ge\,&
(1-v)w^2+
\frac{v\,\langle\nabla w,\nabla v\rangle}{1-v}-\frac{\langle\nabla S(x,t,u),\nabla v\rangle}{(1-v)^2u}\\&\qquad-
\left(k+
\partial_u S(x,t,u)
-\frac{S(x,t,u)}{u}
+\frac{ S(x,t,u)}{(1-v)u}
\right)_+\,w\\
\ge\,&
(1-v)w^2+
\frac{v\,\langle\nabla w,\nabla v\rangle}{1-v}-\frac{\gamma\,|\nabla v|}{(1-v)^2}
-\mu w,
\end{split}
\end{equation*}
as desired.\end{proof}

In the proof of Theorem~\ref{TH2}, we will exploit suitable cut-off
functions. An important property of these auxiliary functions
lies in their precise detachment with respect to the zero level set.
The details of their construction are given in the following result:

\begin{lem}\label{PSI}
Let~$a\in(0,1)$, $R>0$ and $\rho \in (0,R)$. Then, there exists a decreasing
function~$\bar{\psi}\in C^2(\R,[0,1])$ such that
\begin{equation}\label{7132r3yhsjdsid}
{\mbox{$\bar{\psi}(r)=1$ for all~$r\in[0,R-\rho]$, \quad $\bar{\psi}(r)=0$
for all~$r\ge R$,}}
\end{equation}
and, for every~$r\ge0$,
\begin{equation}\label{si-a}
\rho|\bar{\psi}'(r)|+\rho^2|\bar{\psi}''(r)|\le C\big(\bar{\psi}(r)\big)^a,
\end{equation}
for some~$C>0$, depending only
on~$a$.
\end{lem}

\begin{proof}
We introduce
an increasing function~$\alpha\in C^2(\R,[0,1])$ such that:
\begin{equation}\label{ALFA}
\begin{split}&
{\mbox{$\alpha(t)=t^{\frac{2}{1-a}}$ for all~$t\in[0,1/4]$,}}\\&{\mbox{$\alpha(t)=1-(1-t)^4$ for all~$t\in[3/4,1]$,}}\\&{\mbox{$\alpha(t)=0$ for all~$t<0$,}}\\{\mbox{and }}\quad&{\mbox{$\alpha(t)=1$ for all~$t>1$.}}\end{split}\end{equation}
Let also
$$ [0,\infty)\ni r\longmapsto\bar{\psi}(r):=\alpha\left(\frac{R-r}{\rho}\right).$$
We observe that if~$r\in[0,R-\rho]$, then~$\frac{R-r}{\rho}\in\left[1,\frac{R}\rho
\right]$, hence~$\bar{\psi}(r)=1$.
Similarly, if~$r\ge R$, then~$\frac{R-r}{\rho}\le0$ and thus~$\bar{\psi}(r)=0$.
These considerations establish~\eqref{7132r3yhsjdsid}.

Now, we prove~\eqref{si-a}.
For this, in light of~\eqref{7132r3yhsjdsid}, it is enough to
consider the case in which~$r\in[R-\rho,R]$, since otherwise~$\bar{\psi}'(r)=\bar{\psi}''(r)=0$.
Now, when~$r\in[R-\rho,R]$, we have that~$\frac{R-r}{\rho}\in[0,1]$,
and we distinguish two cases:
\begin{eqnarray}
\label{CGAla:1}&&{\mbox{either }}\, \frac{R-r}{\rho}\in\left[0,\frac14\right],
\\
\label{CGAla:2}&&{\mbox{or }}\, \frac{R-r}{\rho}\in\left(\frac14,1\right].
\end{eqnarray}
Suppose first that~\eqref{CGAla:1} is satisfied. Then,~$\bar{\psi}(r)=
\alpha\left(\frac{R-r}{\rho}\right)=\left(\frac{R-r}{\rho}\right)^{\frac{2}{1-a}}$, and accordingly
\begin{eqnarray*}
&& \rho|\bar{\psi}'(r)|+\rho^2|\bar{\psi}''(r)|=
\frac{2}{1-a}\left(\frac{R-r}{\rho}\right)^{\frac{1+a}{1-a}}+
\frac{2(1+a)}{(1-a)^2}\left(\frac{R-r}{\rho}\right)^{\frac{2a}{1-a}}\\
&&\qquad\qquad=\left[
\frac{2}{1-a}\;\frac{R-r}{\rho}+
\frac{2(1+a)}{(1-a)^2}\right]\left(\frac{R-r}{\rho}\right)^{\frac{2a}{1-a}}\\
&&\qquad\qquad\le\left[
\frac{1}{2(1-a)}+
\frac{2(1+a)}{(1-a)^2}\right]\,\big(\bar{\psi}(r)\big)^a.
\end{eqnarray*}
This proves~\eqref{si-a} in this case, and we now suppose that~\eqref{CGAla:2}
holds true. In this situation, we exploit the monotonicity of~$\alpha$ to see that
$$ \bar{\psi}(r)\ge \alpha\left(\frac14\right)=\left(\frac14\right)^{\frac{2}{1-a}}.$$
As a consequence,
\begin{eqnarray*}&& \rho|\bar{\psi}'(r)|+\rho^2|\bar{\psi}''(r)|
=\left| \alpha'\left(\frac{R-r}{\rho}\right)\right|
+\left| \alpha''\left(\frac{R-r}{\rho}\right)\right|\\&&\qquad\quad\le
2\|\alpha\|_{C^2(\R)}\le 2\,4^{\frac{2a}{1-a}}\,\|\alpha\|_{C^2(\R)}\,\big(\bar{\psi}(r)\big)^a.\end{eqnarray*}
This ends the proof of~\eqref{si-a}, as desired.
\end{proof}

As a simple variant of Lemma~\ref{PSI}, we also provide the details of
an auxiliary cut-off function in the time variable:

\begin{lem}\label{ilLECAS2M}
Let~$t_0\in\R$ and~$T>0$.
Let~$a\in(0,1)$ and~$\delta\in(0,T)$. Then, there exists an increasing
function~$\phi\in C^2(\R,[0,1])$ such that
\begin{equation}\label{7132r3yhsjdsid-t}
{\mbox{$\phi(t)=0$ for all~$t\le t_0-T$, and~$\phi(t)=1$
for all~$t\ge t_0-T+\delta$,}}
\end{equation}
and, for every~$t\in\R$,
\begin{equation}\label{si-a-t}
\delta|\phi'(t)|\le C\big(\phi(t)\big)^{\frac{1+a}2},
\end{equation}
for some~$C>0$, depending only
on~$a$.
\end{lem}

\begin{proof} Let~$\alpha$ be the function in~\eqref{ALFA} and define
$$ \phi(t):=\alpha\left(\frac{t-t_0+T}\delta\right).$$
Then, if~$t\le t_0-T$ we have that~$\frac{t-t_0+T}\delta\le0$ and thus~$\phi(t)=0$.
Similarly, if~$t\ge t_0-T+\delta$ then~$\frac{t-t_0+T}\delta\ge1$ and therefore~$\phi(t)=1$.
This proves~\eqref{7132r3yhsjdsid-t}.

To check \eqref{si-a-t}, we can suppose that~$t\in[t_0-T,\,t_0-T+\delta]$ (otherwise~$\phi'(t)$ and the claim
is obviously true). We distinguish two cases,
\begin{eqnarray}
&&{\mbox{either }} \, t\in\left[t_0-T,\,t_0-T+\frac\delta4\right]\label{UAH:1}\\
&&{\mbox{or }} \, t\in\left[t_0-T+\frac\delta4,\,t_0-T+\delta\right].\label{UAH:2}
\end{eqnarray}
If~\eqref{UAH:1} holds true, we have that
$$ \phi(t)=\left(\frac{t-t_0+T}\delta\right)^{\frac2{1-a}},$$
and therefore
$$ |\phi'(t)|=\frac2{(1-a)\delta}
\left(\frac{t-t_0+T}\delta\right)^{\frac{1+a}{1-a}}=\frac2{(1-a)\delta}\big(\phi(t)\big)^{\frac{1+a}{2}},$$
giving~\eqref{si-a-t} in this case.

If instead~\eqref{UAH:2} holds true, we use the monotonicity of~$\alpha$ to write that,
for every~$t\in\left[t_0-T+\frac\delta4,\,t_0-T+\delta\right]$,
$$ \phi(t)\ge \alpha\left(\frac{1}4\right)=\frac1{4^{\frac2{1-a}}},$$
and consequently
$$ |\phi'(t)|=\frac1\delta\left|\alpha'\left(\frac{t-t_0+T}\delta\right)\right|\le\frac{\|\alpha\|_{C^1(\R)}}{\delta}\le\frac{4^{\frac{1+a}{1-a}}\|\alpha\|_{C^1(\R)}}{\delta}\big(\phi(t)\big)^{\frac{1+a}{2}},$$
so that the proof of~\eqref{si-a-t} is concluded.
\end{proof}

To complete the proof of Theorem~\ref{TH2}, we now distinguish four
regimes, according to the cut-off functions in~\eqref{161}.
The estimates in each of these regimes will be dealt with
in the forthcoming Lemmata~\ref{CB:1}, \ref{CB:2}, \ref{CB:3}
and~\ref{CB:4}. To this end, it is also useful to point out the identity
(valid for all smooth and positive functions~$\psi$),
\begin{equation}\label{KASJJSnc}
\frac{\Delta(w\psi)-(w\psi)_t}{2}-\frac{
\langle\nabla(w\psi),\nabla\psi\rangle
}{\psi}=\frac{(\Delta w-w_t)\,\psi}{2}+\frac{(\Delta\psi-\psi_t)\,w}2
-\frac{w\,|\nabla\psi|^2}\psi
.
\end{equation}
Hence, subtracting~$\frac{v\,\langle\nabla v,\nabla(w\psi)\rangle}{1-v}$
to both sides of~\eqref{KASJJSnc},
\begin{equation*}\begin{split}&
\frac{\Delta(w\psi)-(w\psi)_t}{2}-
\left\langle\nabla(w\psi),\frac{\nabla\psi}{\psi}+\frac{v\,\nabla v}{1-v}\right\rangle
\\=\;&\frac{(\Delta w-w_t)\,\psi}{2}+\frac{(\Delta\psi-\psi_t)\,w}2
-\frac{w\,|\nabla\psi|^2}\psi-\frac{v\,\langle\nabla v,\nabla(w\psi)\rangle}{1-v}
.
\end{split}\end{equation*}
Whence it follows from Lemma~\ref{PI} that
\begin{equation}\label{GAH:00PRE}
\begin{split}
&\frac{\Delta(w\psi)-(w\psi)_t}{2}-
\left\langle\nabla(w\psi),\frac{\nabla\psi}{\psi}+\frac{v\,\nabla v}{1-v}\right\rangle
\\ \ge\,&
(1-v)\,w^2 \psi+
\frac{v\psi\,\langle\nabla w,\nabla v\rangle}{1-v}-\frac{\gamma\psi\,|\nabla v|}{(1-v)^2}
-\mu w\psi\\&\qquad+\frac{(\Delta\psi-\psi_t)\,w}2
-\frac{w\,|\nabla\psi|^2}\psi-\frac{v\,\langle\nabla v,\nabla(w\psi)\rangle}{1-v}.
\end{split}
\end{equation}
One can also notice that
$$
\frac{v\psi\,\langle\nabla w,\nabla v\rangle}{1-v}
-\frac{v\,\langle\nabla v,\nabla(w\psi)\rangle}{1-v}=-
\frac{vw\,\langle\nabla \psi,\nabla v\rangle}{1-v},
$$
and thus rewrite~\eqref{GAH:00PRE} in the form
\begin{equation}\label{GAH:00-lapreced}
\begin{split}
&\frac{\Delta(w\psi)-(w\psi)_t}{2}-
\left\langle\nabla(w\psi),\frac{\nabla\psi}{\psi}+\frac{v\,\nabla v}{1-v}\right\rangle
\\ \ge\,&
(1-v)\,w^2 \psi-\frac{\gamma\psi\,|\nabla v|}{(1-v)^2}
-\mu w\psi\\&\qquad+\frac{(\Delta\psi-\psi_t)\,w}2
-\frac{w\,|\nabla\psi|^2}\psi-
\frac{vw\,\langle\nabla \psi,\nabla v\rangle}{1-v}.
\end{split}
\end{equation}
In addition, from~\eqref{DE w}
and Young's inequality with exponents~$4$ and~$4/3$,
\begin{equation}\label{Inqw833222}\begin{split}
&\frac{\gamma\psi\,|\nabla v|}{(1-v)^2}=
\frac{\gamma\psi\,\sqrt{w}}{1-v}=\sqrt[4]{1-v}\;\sqrt{w}\;\sqrt[4]{\psi}\;
\frac{\gamma\psi^{\frac34}}{(1-v)^{\frac54}}
\\&\qquad\le
\frac{1}{4}(1-v)w^2\psi
+\frac{C\gamma^{4/3}\,\psi}{(1-v)^{5/3}},
\end{split}\end{equation}
for some~$C>0$.

Similarly, the use of~\eqref{DE w} and of the
Young's inequality with exponents~$4/3$ and~$4$
gives that, in the support of~$\psi$,
\begin{equation}\label{232:232}\begin{split}&\left|
\frac{vw\,\langle\nabla \psi,\nabla v\rangle}{1-v}\right|\le
\frac{|v|\,w\,|\nabla \psi|\;|\nabla v|}{1-v}=
|v|\,w^{3/2}\,|\nabla \psi|
\\&\qquad=
\left[\left(\frac23\right)^{3/4}\,(1-v)^{3/4}\,w^{3/2}\,\psi^{3/4}\right]
\;\left[\left(\frac32\right)^{3/4}\,\frac{|v|\,|\nabla \psi|}{
(1-v)^{3/4}\, \psi^{3/4}}\right]\\&\qquad \le
\frac14\,(1-v)\,w^{2}\,\psi+
\frac{C\,|v|^4\,|\nabla \psi|^4}{(1-v)^{3}\, \psi^{3}},
\end{split}
\end{equation}
up to renaming~$C>0$.

In light of~\eqref{Inqw833222}
and~\eqref{232:232}, we deduce from~\eqref{GAH:00-lapreced}
\begin{equation}\label{GAH:00-bijmwef9494}
\begin{split}
&\frac{\Delta(w\psi)-(w\psi)_t}{2}-
\left\langle\nabla(w\psi),\frac{\nabla\psi}{\psi}+\frac{v\,\nabla v}{1-v}\right\rangle
\\ \ge\,&\frac{
(1-v)\,w^2 \psi}4-\frac{C\gamma^{4/3}\,\psi}{(1-v)^{5/3}}
-\mu w\psi\\&\qquad+\frac{(\Delta\psi-\psi_t)\,w}2
-\frac{w\,|\nabla\psi|^2}\psi-
\frac{C\,|v|^4\,|\nabla \psi|^4}{(1-v)^{3}\, \psi^{3}}.
\end{split}
\end{equation}
Besides, by the Cauchy-Schwarz inequality,
\begin{equation}\label{FGgafmeg801} \mu w\psi=\big(\sqrt{1-v}\;w\;\sqrt{\psi}\big)
\;\left(\frac{\mu\, \sqrt{\psi}}{\sqrt{1-v}}\right)\le
\frac{(1-v)w^2\psi}{8}+
\frac{C\mu^2\psi}{{1-v}},
\end{equation}
up to renaming~$C$,
which, combined with~\eqref{GAH:00-bijmwef9494}, proves that
\begin{equation}\label{GAH:00}
\begin{split}
&\frac{\Delta(w\psi)-(w\psi)_t}{2}-
\left\langle\nabla(w\psi),\frac{\nabla\psi}{\psi}+\frac{v\,\nabla v}{1-v}\right\rangle
\\ \ge\,&\frac{
(1-v)\,w^2 \psi}8-\frac{C\gamma^{4/3}\,\psi}{(1-v)^{5/3}}
-
\frac{C\mu^2\psi}{{1-v}}
\\&\qquad+\frac{(\Delta\psi-\psi_t)\,w}2
-\frac{w\,|\nabla\psi|^2}\psi-
\frac{C\,|v|^4\,|\nabla \psi|^4}{(1-v)^{3}\, \psi^{3}}.
\end{split}
\end{equation}
We will use \eqref{GAH:00} as a pivotal inequality
in the forthcoming computations.
We have:

\begin{lem}\label{CB:1}
In the setting of Theorem~\ref{TH2},
in $B(x_0,R-\rho) \times [t_0-T,t_0]$ it holds
\begin{equation}\label{CA:0}
w \le\left[\tau_u^2+C\left(\gamma^{2/3}
+\mu +
\frac{1}{\rho^2}+\frac{1 }{ \rho(R-\rho)}
+\frac{\sqrt{ k_+ }}{\rho}\right)\right],
\end{equation}
for some~$C>0$.
\end{lem}

\begin{proof}
Let~$a\in(0,1)$, to be conveniently chosen in what follows.
For every~$x\in B(x_0,R)$,
we define \begin{equation}\label{PSIDEFINI}\psi(x):=\bar{\psi}(d(x,x_0)),\end{equation}
where~$d(\cdot,\cdot)$ represents the geodesic distance
and~$\bar{\psi}$ is the function introduced in Lemma~\ref{PSI}.

By~\eqref{ricci k}, we know that
\begin{equation*}
\Delta d(x,x_0)\le\frac{n-1}{d(x,x_0)}+\sqrt{(n-1)k_+},
\end{equation*}
see e.g.~\cite[formula (2.1)]{4CM} and the references therein.

This and~\eqref{si-a} entail that
\begin{equation}\label{2:30}
\begin{split}&
|\nabla\psi(x)|= |\bar{\psi}'(d(x,x_0))\,\nabla d(x,x_0)|\le
\frac{C\big( \psi(x)\big)^a}\rho\\{\mbox{and }}\quad&-\Delta\psi(x)=
-\bar{\psi}'(d(x,x_0))\Delta d(x,x_0)-
\bar{\psi}''(d(x,x_0))|\nabla d(x,x_0)|^2\\&\qquad\qquad\le
\frac{C\big( \psi(x)\big)^a}\rho\,\left(\frac{n-1}{d(x,x_0)}+\sqrt{(n-1)k_+}\right)+
\frac{C\big( \psi(x)\big)^a}{\rho^2}.
\end{split}\end{equation}
Now, we consider~$\tilde w:=w\psi$ and,
in the support of~$\psi$, we can exploit~\eqref{GAH:00}
and find that
\begin{equation}\label{GAH203e}
\begin{split}
&\frac{\Delta\tilde w-\tilde w_t}{2}-
\left\langle\nabla \tilde w,\frac{\nabla\psi}{\psi}+\frac{v\,\nabla v}{1-v}\right\rangle
\\ \ge\,&\frac{
(1-v)\,w^2 \psi}8-\frac{C\gamma^{4/3}\,\psi}{(1-v)^{5/3}}
-\frac{C\mu^2\psi}{{1-v}}\\&\qquad+\frac{\Delta\psi\,w}2
-\frac{w\,|\nabla\psi|^2}\psi-\frac{C\,|v|^4\,|\nabla \psi|^4}{(1-v)^{3}\, \psi^{3}}.
\end{split}
\end{equation}
We take~$(x_1, t_1)$ in the closure of~$Q_{R,T}$
realizing the maximum of~$\tilde w$.
Since~$\tilde w(x,t)=0$ if~$x\in\partial B(x_0,R)$,
necessarily~$x_1$ is an interior point of~$B(x_0,R)$.
Consequently~$\nabla\tilde w(x_1,t_1)=0$ and~$\Delta \tilde w(x_1,t_1)\le0$.
Hence, inserting this information into~\eqref{GAH203e},
we obtain that
\begin{equation}\label{83rxfu}
\begin{split}&
0 \ge\frac{\tilde w_t}{2}+\frac{
(1-v)\,w^2 \psi}8-\frac{C\gamma^{4/3}\,\psi}{(1-v)^{5/3}}
-\frac{C\mu^2\psi}{{1-v}}\\&\qquad\qquad+\frac{\Delta\psi\,w}2
-\frac{w\,|\nabla\psi|^2}\psi-
\frac{C\,|v|^4\,|\nabla \psi|^4}{(1-v)^{3}\, \psi^{3}}
\Bigg|_{(x,t)=(x_1,t_1)}.
\end{split}
\end{equation}
We now distinguish two cases,
\begin{eqnarray}
\label{CA:1}&&{\mbox{either \, $t_1=t_0-T$,}}
\\&&{\mbox{or \, $t_1\in (t_0-T,t_0]$.}}\label{CA:2}
\end{eqnarray}
Suppose first that~\eqref{CA:1} holds true.
Then, for every~$(x,t)\in Q_{R,T}$,
\begin{equation}\nonumber
\aligned
\tilde w (x,t)&\leq \tilde w(x_1,t_0-T)\\&\leq
\sup_{x\in B(x_0,R)} \tilde w(x,t_0-T)\\&\leq\sup_{x\in B(x_0,R)}w(x,t_0-T)\\&=\sup_{x\in B(x_0,R)}\left ( \frac{|\nabla u|^2}{u^2(1-v)^2}\right )(x,t_0-T)
\\&\leq \tau_u^2,
\endaligned
\end{equation}
thanks to~\eqref{sigtau} and~\eqref{DE w}.
In particular, for all~$(x,t)\in B(x_0,R-\rho) \times[t_0-T,t_0]$,
$$ w(x,t)=\tilde w (x,t)\le\tau_u^2,$$
and this proves~\eqref{CA:0} in this case.

Hence, to complete the proof of~\eqref{CA:0}, we now consider the
case in which~\eqref{CA:2} is satisfied. Then, $\tilde w_t(x_1,t_1)\ge0$,
and consequently~\eqref{83rxfu} entails that
\begin{equation}\begin{split}\label{21e5}
&0 \ge\frac{
(1-v)\,w^2 \psi}8-\frac{C\gamma^{4/3}\,\psi}{(1-v)^{5/3}}
-\frac{C\mu^2\psi}{{1-v}}+\frac{\Delta\psi\,w}2
\\&\qquad\qquad-\frac{w\,|\nabla\psi|^2}\psi-
\frac{C\,|v|^4\,|\nabla \psi|^4}{(1-v)^{3}\, \psi^{3}}
\Bigg|_{(x,t)=(x_1,t_1)}.
\end{split}\end{equation}
Our goal is now to estimate the terms in~\eqref{21e5}
that contain powers of~$w$ strictly less than~$2$, in order to ``reabsorb'' them into
the quadratic term.
To this end, exploiting~\eqref{2:30}
inside~\eqref{21e5}, and renaming~$C>0$ (possibly depending
on~$a$), we find that
\begin{equation}\begin{split}\label{21e5:BIS}
&\frac18\,(1-v)\,w^2 \psi\Bigg|_{(x,t)=(x_1,t_1)} \\
&\le
\frac{C\gamma^{4/3}\,\psi}{(1-v)^{5/3}}
+\frac{C\mu^2\psi}{{1-v}}-\frac{\Delta\psi\,w}2
+\frac{Cw\psi^{2a-1}}{\rho^2}+
\frac{C\,|v|^4\,\psi^{4a-3}}{(1-v)^{3}\, \rho^4}\Bigg|_{(x,t)=(x_1,t_1)}.
\end{split}\end{equation}
Also, from~\eqref{2:30},
\begin{eqnarray*}
-\frac{\Delta\psi\,w}2&\le&\frac{Cw}{2}\,\left[
\frac{\psi^a }{ \rho}\,\left(\frac{n-1}{d}+\sqrt{(n-1)k_+}\right)+
\frac{\psi^a }{ \rho^2}\right]\\
&=&
\frac{C\,\sqrt{1-v}\;w\;\sqrt{\psi}}{2}\,\left[
\rho\,\left(\frac{n-1}{d}+\sqrt{(n-1)k_+}\right)+1
\right]\frac{\psi^{a-\frac12} }{ \sqrt{1-v}\;\rho^2}
\end{eqnarray*}
where~$d:=d(x,x_0)$.
Then, by the Cauchy-Schwarz inequality,
\begin{equation}\label{Qua4d}-\frac{\Delta\psi\,w}2\le
\frac{(1-v)w^2\psi}{16}+C\left[
\rho\,\left(\frac{n-1}{d}+\sqrt{(n-1)k_+}\right)+1
\right]^2\frac{\psi^{2a-1} }{ (1-v)\rho^4}
,\end{equation}
up to renaming~$C$.

We also remark that when~$x\in B(x_0,R-\rho)$, then~$d=d(x,x_0)\in[0,R-\rho)$,
and thus~$\frac{R-d}{\rho}>1$, which gives that
$$ \psi(x)=\bar{\psi}(d)=\alpha\left(\frac{R-d}{\rho}\right)=1.$$
In particular,
\begin{equation} \label{781342536475:0293458667676644}
\Delta\psi(x)=0\qquad{\mbox{ for all }}x\in B(x_0,R-\rho).\end{equation}
Now we claim that
\begin{equation}\label{98172364098765efgdf5uyhji345677}
-\frac{\Delta\psi\,w}2
\,\le\,
\frac{(1-v)w^2\psi}{16}+
\frac{C\psi^{2a-1} }{ (1-v)\rho^2(R-\rho)^2}
+\frac{C k_+\psi^{2a-1} }{ (1-v)\rho^2}+
\frac{C\psi^{2a-1} }{ (1-v)\rho^4}
,\end{equation}
up to renaming~$C$.
To prove this, we first observe that in~$B(x_0,R-\rho)$
the claim is obvious, thanks to~\eqref{781342536475:0293458667676644}.
Hence, we can focus
on the complement of~$B(x_0,R-\rho)$,
where
\begin{equation}\label{KA7281323tf7124r53734t9ghghghhgh24356}
d\ge R-\rho.\end{equation}
Then, we can exploit~\eqref{Qua4d}, combined with~\eqref{KA7281323tf7124r53734t9ghghghhgh24356},
and obtain~\eqref{98172364098765efgdf5uyhji345677},
as desired.

Thus, we insert~\eqref{98172364098765efgdf5uyhji345677} into~\eqref{21e5:BIS}
and we obtain that
\begin{equation}\label{39029735-3}\begin{split}&
\frac1{16}\,(1-v)\,w^2 \psi\Bigg|_{(x,t)=(x_1,t_1)} \le
\frac{C\gamma^{4/3}\,\psi}{(1-v)^{5/3}}
+\frac{C\mu^2\psi}{{1-v}}
+\frac{Cw\psi^{2a-1}}{\rho^2}\\&\qquad\qquad+
\frac{C|v|^4\,\psi^{4a-3}}{(1-v)^{3}\, \rho^4}+\frac{C\psi^{2a-1} }{ (1-v)\rho^2(R-\rho)^2}
+\frac{C k_+\psi^{2a-1} }{ (1-v)\rho^2}+
\frac{C\psi^{2a-1} }{ (1-v)\rho^4}\Bigg|_{(x,t)=(x_1,t_1)}.\end{split}
\end{equation}
Furthermore, using the Cauchy-Schwarz inequality,
\begin{equation}\label{23423435AYH}
\begin{split}&
\frac{Cw\psi^{2a-1}}{\rho^2}=
\sqrt{1-v}\;w\;\sqrt{\psi}\;
\frac{C\psi^{2a-\frac32}}{\sqrt{1-v}\;\rho^2}
\le\frac1{32}
(1-v)w^2\psi+\frac{
C\psi^{4a-3}}{(1-v)\rho^4},
\end{split}\end{equation}
up to renaming~$C$, which together with~\eqref{39029735-3} entails that
at the point~$(x,t)=(x_1,t_1)$
\begin{equation}\label{39029735-6}\begin{split}&
\frac1{32}\,(1-v)\,w^2 \psi \le
\frac{C\gamma^{4/3}\,\psi}{(1-v)^{5/3}}
+\frac{C\mu^2\psi}{{1-v}}+\frac{
C\psi^{4a-3}}{(1-v)\rho^4}\\&\qquad\qquad+
\frac{C|v|^4\,\psi^{4a-3}}{(1-v)^{3}\, \rho^4}+\frac{C\psi^{2a-1} }{ (1-v)\rho^2(R-\rho)^2}
+\frac{C k_+\psi^{2a-1} }{ (1-v)\rho^2}+
\frac{C\psi^{2a-1} }{ (1-v)\rho^4}.\end{split}
\end{equation}
Now, up to renaming~$C$, and recalling the maximizing property of~$(x_1,t_1)$,
we can rewrite~\eqref{39029735-6} in case~\eqref{CA:2}
in the form
\begin{equation}\label{39029735-7}\begin{split}
\sup_{Q_{R,T}}w^2 \psi^2 \le\;&
\frac{C\gamma^{4/3}\,\psi^2}{(1-v)^{8/3}}
+\frac{C\mu^2\psi^2}{{(1-v)^2}}+\frac{
C\psi^{4a-2}}{(1-v)^2\rho^4}+
\frac{C|v|^4\,\psi^{4a-2}}{(1-v)^{4}\, \rho^4}\\&\qquad
+\frac{C\psi^{2a} }{ (1-v)^2\rho^2(R-\rho)^2}
+\frac{C k_+\psi^{2a} }{ (1-v)^2\rho^2}+
\frac{C\psi^{2a} }{ (1-v)^2\rho^4}\Bigg|_{(x,t)=(x_1,t_1)}.\end{split}
\end{equation}
Now we choose~$a:=1/2$, and we recall that~$0\le\psi\le1$ and that~$\psi=1$
in~$B(x_0,R-\rho)$. In this way, in case~\eqref{CA:2} we deduce from~\eqref{39029735-7} that
\begin{equation}\label{39029735-8}\begin{split}&\sup_{B(x_0,R-\rho)
\times(t_0-T,t_0]
}w^2 =
\sup_{B(x_0,R-\rho)
\times(t_0-T,t_0]
}w^2 \psi^2 \\&\quad\le
\frac{C\gamma^{4/3} }{(1-v)^{8/3}}
+\frac{C\mu^2 }{{(1-v)^2}}+\frac{
C}{(1-v)^2\rho^4}\\&\qquad\qquad+
\frac{C|v|^4}{(1-v)^{4}\, \rho^4}+\frac{C }{ (1-v)^2\rho^2(R-\rho)^2}
+\frac{C k_+ }{ (1-v)^2\rho^2}\Bigg|_{(x,t)=(x_1,t_1)}
.\end{split}
\end{equation}
Up to renaming~$C$, we can also rewrite~\eqref{39029735-8}
as
\begin{equation*}
\begin{split}&\sup_{B(x_0,R-\rho)
\times(t_0-T,t_0]
}w^2 \le
\frac{C\gamma^{4/3} }{(1-v)^{8/3}}
+\frac{C\mu^2 }{{(1-v)^2}}\\&\qquad\qquad+
\frac{C(1+v^4)}{(1-v)^{4}\, \rho^4}+\frac{C }{ (1-v)^2\rho^2(R-\rho)^2}
+\frac{C k_+ }{ (1-v)^2\rho^2}\Bigg|_{(x,t)=(x_1,t_1)}.\end{split}
\end{equation*}
That is, recalling~\eqref{Jna:0},
\begin{equation*}
\sup_{B(x_0,R-\rho)
\times(t_0-T,t_0]}w^2 \le
C\gamma^{4/3}
+C\mu^2+
\frac{C}{ \rho^4}+\frac{C }{ \rho^2(R-\rho)^2}
+\frac{C k_+ }{\rho^2},
\end{equation*}
which establishes~\eqref{CA:0}, as desired.
\end{proof}

\begin{lem}\label{CB:2}
In the setting of Theorem~\ref{TH2},
in~$B(x_0,R)\times [t_0-T+\delta,t_0]$ it holds
\begin{equation}\label{C3456dffA:0}
w\le\left[\sigma_u^2+C\left(\gamma^{2/3}+
\mu+\frac{1}{\delta }\right)\right],
\end{equation}
for some~$C>0$.
\end{lem}

\begin{proof} We take~$\phi$ as in Lemma~\ref{ilLECAS2M} (say, with~$a:=1/2$),
and we define~$\tilde w(x,t):=w(x,t)\phi(t)$. Then, in light of~\eqref{GAH:00},
\begin{equation}\label{wefdgGAH:00}
\begin{split}
\frac{\Delta\tilde w-\tilde w_t}{2}-
\left\langle\nabla\tilde w,\frac{v\,\nabla v}{1-v}\right\rangle
\ge\frac{
(1-v)\,w^2 \phi}8-\frac{C\gamma^{4/3}\,\psi}{(1-v)^{5/3}}
-\frac{C\mu^2\psi}{{1-v}} -\frac{\phi_t\,w}2.
\end{split}
\end{equation}
Suppose that the maximum of $\tilde w$ in the closure of~$Q_{R,T}$ is reached at $(x_1, t_1)$.
Since~$\tilde w=0$ when~$t=t_0-T$, we know that~$t_1\in(t_0-T,t_0]$.
As a result,
\begin{equation}\label{wefdgGAH:001}
\tilde w_t(x_1,t_1)\ge0.
\end{equation}
We then distinguish two cases,
\begin{eqnarray}
&&\label{SPALS:1}{\mbox{either }} \, x_1\in\partial B(x_0,R),\\
&&\label{SPALS:2}{\mbox{or }} \,
x_1\in B(x_0,R).
\end{eqnarray}
If~\eqref{SPALS:1} holds true, then, in~$ Q_{R,T}$,
\begin{equation}\nonumber
\aligned
\tilde w\leq \tilde w(x_1,t_1)&\leq
\sup_{{x\in\partial B(x_0,R)}\atop{t\in[t_0-T,t_0]}}\tilde w(x,t)\\
&\leq\sup_{{x\in\partial B(x_0,R)}\atop{t\in[t_0-T,t_0]}}w(x,t)\\
&=\sup_{{x\in\partial B(x_0,R)}\atop{t\in[t_0-T+\delta/2,t_0]}}
\Big|\frac{|\nabla u|^2}{u^2(1-v)^2}\Big|(x,t)\\&=\sigma_u^2,
\endaligned
\end{equation}
due to~\eqref{sigtau}
and~\eqref{DE w}.
In particular, since~$\phi=1$ for any~$t\ge t_0-T+\delta$, we have that
$$ w=\tilde{w}\le \sigma_u^2$$
in~$B(x_0,R)\times [t_0-T+\delta,t_0]$.

This proves~\eqref{C3456dffA:0} in this case, and we now suppose that~\eqref{SPALS:2}
holds true. Then, we have that~$\Delta \tilde w(x_1,t_1)\leq0$ and~$\nabla \tilde w(x_1,t_1)=0$.
These observations and~\eqref{wefdgGAH:001}, together with~\eqref{wefdgGAH:00},
yield that
\begin{equation}\label{wefdgGAH:003}
\begin{split}
0
\ge\frac{
(1-v)\,w^2 \phi}8-\frac{C\gamma^{4/3}\,\psi}{(1-v)^{5/3}}
-\frac{C\mu^2\psi}{{1-v}} -\frac{\phi_t\,w}2\Bigg|_{(x,t)=(x_1,t_1)}.
\end{split}
\end{equation}
Moreover, by~\eqref{si-a-t} and the Cauchy-Schwarz inequality,
\begin{equation}\label{we555fd666345678gGAH:003} \frac{\phi_t\,w}2\le
\frac{C\,\phi^{\frac{1+a}2}\,w}{2\delta}
=\frac{\sqrt{1-v}\;w\;\sqrt{\phi}}4\;
\frac{2C\phi^{\frac{a}2}}{\delta\sqrt{1-v}}\le
\frac{(1-v)w^2\phi}{16}+
\frac{C\phi^a}{\delta^2(1-v)},
\end{equation}
possibly renaming constants.
Plugging this into~\eqref{wefdgGAH:003}, we conclude that
$$
(1-v)\,w^2 \phi\Bigg|_{(x,t)=(x_1,t_1)}\le \frac{C\gamma^{4/3}\,\phi}{
(1-v)^{5/3}}
+\frac{C\mu^2\phi}{{1-v}}+\frac{C\phi^a}{\delta^2(1-v)}
\Bigg|_{(x,t)=(x_1,t_1)} .$$
That is
$$
w^2 \phi\Bigg|_{(x,t)=(x_1,t_1)}\le \frac{C\gamma^{4/3}
\,\phi}{(1-v)^{8/3}}
+\frac{C\mu^2\phi}{(1-v)^2}+\frac{C\phi^a}{\delta^2(1-v)^2}
\Bigg|_{(x,t)=(x_1,t_1)} .$$
Now, since~$0\le\phi\le1$ and~$\phi=1$ for any~$t\ge t_0-T+\delta$, this
implies that
$$ \sup_{B(x_0,R)\times [t_0-T+\delta,t_0]}w^2=
\sup_{B(x_0,R)\times [t_0-T+\delta,t_0]}w^2\phi^2\le
\frac{C\gamma^{4/3}}{(1-v)^{8/3}}
+\frac{C\mu^2}{(1-v)^2}+\frac{C}{\delta^2(1-v)^2}
\Bigg|_{(x,t)=(x_1,t_1)}.
$$
As a consequence, recalling also~\eqref{Jna:0},
we obtain~\eqref{C3456dffA:0}, as desired.
\end{proof}

\begin{lem}\label{CB:3}
In the setting of Theorem~\ref{TH2},
in~$B(x_0,R)\times[t_0-T,t_0]$ it holds
\begin{equation}\label{C3456543665888885yhdffA:0}
w\le\left[\sigma_u^2+\tau_u^2+
C\left(\gamma^{2/3}+\mu\right)
\right],
\end{equation}
for some~$C>0$.\end{lem}

\begin{proof}
We suppose that the maximum of $w$ in the closure of~$Q_{R,T}$ is reached at the point~$(x_1, t_1)$.  We distinguish three possibilities:
\begin{eqnarray}
\label{CA:3-001} &&{\mbox{either }}\, x_1\in B(x_0,R) {\mbox{ and }}t_1\in(t_0-T,t_0]
,\\
\label{CA:3-002} &&{\mbox{or }}\, x_1\in B(x_0,R) {\mbox{ and }}t_1=t_0-T
,\\
\label{CA:3-003} &&{\mbox{or }}\, x_1\in \partial B(x_0,R) {\mbox{ and }}t_1\in[t_0-T,t_0].
\end{eqnarray}
Suppose first that~\eqref{CA:3-001} holds true. Then, we have that~$\Delta w(x_1,t_1)\le0$,
$\nabla w(x_1,t_1)=0$ and~$w_t(x_1,t_1)\ge0$. Therefore, in light of Lemma~\ref{PI},
we obtain that
\begin{equation}\label{54:1-1}
0\ge
(1-v)\,w^2-\frac{\gamma\,|\nabla v|}{(1-v)^2}
-\mu w\Bigg|_{(x,t)=(x_1,t_1)}.
\end{equation}
We insert~\eqref{Inqw833222} and~\eqref{FGgafmeg801}
(used here with~$\psi:=1$)
into~\eqref{54:1-1} to see that
$$ \frac12\,(1-v)\,w^2\Bigg|_{(x,t)=(x_1,t_1)}\le\frac{C\gamma^{4/3}}{(1-v)^{5/3}}+
\frac{C\mu^2}{{1-v}}\Bigg|_{(x,t)=(x_1,t_1)}.$$
Hence, using the maximality of~$(x_1,t_1)$ and recalling~\eqref{Jna:0},
$$ \sup_{B(x_0,R)\times [t_0-T,t_0]}w^2\le
\frac{C\gamma^{4/3}}{(1-v)^{5/3}}+
\frac{C\mu^2}{{1-v}}\Bigg|_{(x,t)=(x_1,t_1)}
\le C\gamma^{4/3}+
C\mu^2.$$
This proves~\eqref{C3456543665888885yhdffA:0} in this case.
Hence we now assume that~\eqref{CA:3-002} is satisfied. Then, in~$Q_{R,T}$,
\begin{equation}\label{8239tru75766} \begin{split}&
w\le w(x_1,t_0-T)=\frac{|\nabla v(x_1,t_0-T)|^2}{(1-v(x_1,t_0-T))^2}\\&\qquad\qquad=
\frac{|\nabla u(x_1,t_0-T)|^2}{(u(x_1,t_0-T))^2(1-v(x_1,t_0-T))^2}\le\tau_u^2,
\end{split}\end{equation}
thanks to~\eqref{sigtau}, \eqref{DE w} and~\eqref{Jna:1}.

This establishes~\eqref{C3456543665888885yhdffA:0} in this case,
and thus we now suppose that~\eqref{CA:3-003} is satisfied.
In this case, the computation in~\eqref{8239tru75766} gives that,
in~$Q_{R,T}$,
$$
w\le w(x_1,t_1)=\frac{|\nabla u(x_1,t_1)|^2}{(u(x_1,t_1))^2(1-v(x_1,t_1))^2}\le\sigma_u^2,$$
whence the proof of~\eqref{C3456543665888885yhdffA:0} is complete.
\end{proof}

\begin{lem}\label{CB:4}
In the setting of Theorem~\ref{TH2},
in~$B(x_0,R-\rho) \times [t_0-T+\delta,t_0]$ it holds
\begin{equation}\label{C3123456757374254777456dffA:0}
w\le C\,\left(
\gamma^{2/3}+\mu+ \frac{1}{\rho(R-\rho)}+
\frac{\sqrt{k_+}}{\rho}
+\frac{1}{\delta}+
\frac{1}{ \rho^2}
\right),
\end{equation}
for some~$C>0$.\end{lem}

\begin{proof} Let~$a\in(0,1)$ to be conveniently chosen in what follows.
Let also~$\psi$ be as in~\eqref{PSIDEFINI}
and~$\phi$ be as in
Lemma~\ref{ilLECAS2M}. We define~$\Phi(x,t):=\psi(x)\phi(t)$ and~$\tilde w:=w\Phi$.
Suppose that the maximum of $\tilde w$ in the closure of~$Q_{R,T}$ is reached at some point~$(x_1, t_1)$.
Since~$\Phi$ vanishes along the parabolic boundary, we know that~$x_1\in B(x_0,R)$
and~$t_1\in(t_0-T,t_0]$. As a consequence,
\begin{equation}\label{7yhbcampP6tg} \Delta\tilde w(x_1,t_1)\le0,\qquad
\nabla\tilde w(x_1,t_1)=0\qquad{\mbox{and}}\qquad\tilde w_t(x_1,t_1)\ge0.\end{equation}
Combining this information with~\eqref{GAH:00}, we obtain that
\begin{equation}\label{OKAiiot57}
\begin{split}
&0\ge\frac{(1-v)\,w^2 \Phi}8-\frac{C\gamma^{4/3}\,\psi}{(1-v)^{5/3}}
-\frac{C\mu^2\psi}{{1-v}}\\&\qquad+\frac{(\Delta\Phi-\Phi_t)\,w}2
-\frac{w\,|\nabla\Phi|^2}\Phi-\frac{C\,|v|^4\,|\nabla \psi|^4}{(1-v)^{3}\, \psi^{3}}\Bigg|_{(x,t)=(x_1,t_1)}.
\end{split}
\end{equation}
Now, from~\eqref{Qua4d},
\begin{equation}\label{989890-09-01}\begin{split}
-\frac{\Delta\Phi\,w}2
=
-\frac{\phi\Delta\psi\,w}2\le
\frac{(1-v)w^2\Phi}{32}+C\left[
\rho\,\left(\frac{n-1}{d}+\sqrt{(n-1)k_+}\right)+1
\right]^2\frac{\psi^{2a-1} \phi}{ (1-v)\rho^4}
,\end{split}\end{equation}
and, from~\eqref{we555fd666345678gGAH:003},
\begin{equation}  \label{989890-09-02}\frac{\Phi_t\,w}2=
\frac{\psi\phi_t\,w}2\le
\frac{(1-v)w^2\Phi}{32}+
\frac{C\phi^a\psi}{\delta^2(1-v)}.\end{equation}
We plug these items of information into~\eqref{OKAiiot57},
and we find that
\begin{align*}
&\frac1{16}(1-v)\,w^2 \Phi\Bigg|_{(x,t)=(x_1,t_1)}\\
&\le
\frac{C\gamma^{4/3}\Phi}{(1-v)^{5/3}}+\frac{C\mu^2\Phi}{{1-v}}+ C\left[
\rho\,\left(\frac{n-1}{d}+\sqrt{(n-1)k_+}\right)+1
\right]^2\frac{\psi^{2a-1} \phi}{ (1-v)\rho^4}\\
&+\frac{C\phi^a\psi}{\delta^2(1-v)}
+\frac{w\,|\nabla\Phi|^2}\Phi+\frac{C\,|v|^4\,|\nabla \psi|^4}{(1-v)^{3}\, \psi^{3}}\Bigg|_{(x,t)=(x_1,t_1)}.\end{align*}
This and~\eqref{23423435AYH} entail that
\begin{align*} &\frac1{32}(1-v)\,w^2\Phi\Bigg|_{(x,t)=(x_1,t_1)}\\
&\le
\frac{C\gamma^{4/3}\Phi}{(1-v)^{5/3}}+\frac{C\mu^2\Phi}{{1-v}}+ C\left[
\rho\,\left(\frac{n-1}{d}+\sqrt{(n-1)k_+}\right)+1
\right]^2\frac{\psi^{2a-1} \phi}{ (1-v)\rho^4}
\\
&+\frac{C\phi^a\psi}{\delta^2(1-v)}
+\frac{C\psi^{4a-3}\phi}{(1-v)\rho^4}+\frac{C\,|v|^4\,|\nabla \psi|^4}{(1-v)^{3}\, \psi^{3}}\Bigg|_{(x,t)=(x_1,t_1)}.\end{align*}
Using this and~\eqref{2:30} (and adjusting the constants)
we conclude that
\begin{align*} &(1-v)\,w^2 \Phi\Bigg|_{(x,t)=(x_1,t_1)}\\
&\le
\frac{C\gamma^{4/3}\Phi}{(1-v)^{5/3}}+\frac{C\mu^2\Phi}{{1-v}}+ C\left[
\rho\,\left(\frac{n-1}{d}+\sqrt{(n-1)k_+}\right)+1
\right]^2\frac{\psi^{2a-1} \phi}{ (1-v)\rho^4}\\
&+\frac{C\phi^a\psi}{\delta^2(1-v)}
+\frac{C\psi^{4a-3}\phi}{(1-v)\rho^4}+
\frac{C|v|^4\,\psi^{4a-3}\phi}{(1-v)^{3}\, \rho^4}
\Bigg|_{(x,t)=(x_1,t_1)}.\end{align*}
Therefore, choosing~$a:=3/4$,
we see that, for every~$x\in B(x_0,R-\rho)$ and~$t\in [t_0-T+\delta,t_0]$,
\begin{eqnarray*}
&&w^2(x,t)= w^2(x,t)\Phi^2(x,t)=\tilde w^2(x,t)\le
\tilde w^2(x_1,t_1)=
w^2(x_1,t_1)\Phi^2(x_1,t_1)\\&&
\qquad\le\frac{C\gamma^{4/3}}{(1-v)^{8/3}}
+\frac{C\mu^2}{{(1-v)^2}}+ \left[
\rho\,\left(\frac{n-1}{d}+\sqrt{(n-1)k_+}\right)+1
\right]^2\frac{C }{ (1-v)^2\,\rho^4}
\\&&\qquad\qquad+\frac{C}{\delta^2(1-v)^2}
+\frac{C}{(1-v)^2\,\rho^4}+
\frac{C|v|^4}{(1-v)^{4}\, \rho^4}\Bigg|_{(x,t)=(x_1,t_1)}
,
\end{eqnarray*}
that, recalling~~\eqref{Jna:0}, yields the desired estimate in~\eqref{C3123456757374254777456dffA:0}.
\end{proof}

Now we use the notation
\begin{equation}\label{TUTTITE}
\begin{split}
&{\widetilde{\mathscr{C}}}:=
\gamma^{2/3}+\mu,\\
&{\widetilde{\mathscr{T}}}:=\frac{1}{\delta},\\
{\mbox{and }}\quad&{\widetilde{\mathscr{S}}}:=\frac{1}{\rho^2}+\frac{1}{\rho(R-\rho)}
+\frac{\sqrt{k_+}}{\rho}.
\end{split}
\end{equation}
In this setting, the term~${\widetilde{\mathscr{C}}}$ in~\eqref{TUTTITE}
denotes a common quantity for our main estimates, while
the terms~${\widetilde{\mathscr{T}}}$ and~${\widetilde{\mathscr{S}}}$ play the role of parabolic boundary terms
in time and in space respectively. As a matter of fact, by combining Lemmata~\ref{CB:1},
\ref{CB:2}, \ref{CB:3} and~\ref{CB:4} we find that

\begin{cor}\label{E67QRqwer3456788gfnh}
In the setting of Theorem~\ref{TH2}, the function~$w$ can be
estimated by
\begin{eqnarray*}
C{\widetilde{\mathscr{C}}}+
\big(\tau_u^2+C{\widetilde{\mathscr{S}}}\big)&&
{\mbox{in~$ B(x_0,R-\rho)\times [t_0-T,t_0]$,}}\\C{\widetilde{\mathscr{C}}}+
\big(\sigma_u^2+ C{\widetilde{\mathscr{T}}}
\big)&&
{\mbox{in~$ B(x_0,R)\times [t_0-T+\delta,t_0]$,}}\\C{\widetilde{\mathscr{C}}}+
\big(\sigma_u^2+\tau_u^2
\big)&&
{\mbox{in~$B(x_0,R)\times[t_0-T,t_0]$,}}
\\C{\widetilde{\mathscr{C}}}+
C\,\big( {\widetilde{\mathscr{S}}}
+{\widetilde{\mathscr{T}}}
\big)&&{\mbox{in~$
B(x_0,R-\rho)\times [t_0-T+\delta,t_0]$.}}
\end{eqnarray*}
for some~$C>0$.
\end{cor}

Hence, considering the more convenient term in any common domain, we deduce
from Corollary~\ref{E67QRqwer3456788gfnh} that:

\begin{cor}
In the setting of Theorem~\ref{TH2}, at any point in~$ Q_{R,T}$, we have that
\begin{equation}\label{C31234567xccgtt55d57374254777456dffA:0}\begin{split}
&w\le
C{\widetilde{\mathscr{C}}}+
\Big[\min\left\{\sigma_u^2+\tau_u^2,\,
\sigma_u^2+ C{\widetilde{\mathscr{T}}},\,
\tau_u^2+ C{\widetilde{\mathscr{S}}},\,
C({\widetilde{\mathscr{T}}}+{\widetilde{\mathscr{S}}})\right\}
\chi_{B(x_0,R-\rho)\times [t_0-T+\delta,t_0]}\\
&\qquad\qquad+\left(\sigma_u^2 +\min\left\{\tau_u^2,\,C{\widetilde{\mathscr{T}}}\right\}\right)\chi_{(B(x_0,R)\setminus B(x_0,R-\rho))\times [t_0-T+\delta,t_0]}\\
&\qquad\qquad+
\left(\tau_u^2 +\min\left\{\sigma_u^2,\,C{\widetilde{\mathscr{S}}}\right\}\right)
\chi_{B(x_0,R-\rho)\times [t_0-T,t_0-T+\delta]}\\
&\qquad\qquad+\left(\sigma_u^2 +\tau_u^2\right)\chi_{(B(x_0,R)\setminus B(x_0,R-\rho))\times [t_0-T,t_0-T+\delta]}
\Big],\end{split}
\end{equation}
for some~$C>0$.
\end{cor}

\begin{proof}[Completion of the proof of Theorem~\ref{TH2}]
In light of~\eqref{DE w} and~\eqref{Jna:1},
$$ w=\frac{|\nabla v|^2}{(1-v)^{2}}=\frac{|\nabla u|^2}{u^2(1-v)^{2}}.
$$
This and~\eqref{C31234567xccgtt55d57374254777456dffA:0} give that
\begin{equation*}\begin{split}
&
\frac{|\nabla u|^2}{u^2(1-v)^{2}}\le
C{\widetilde{\mathscr{C}}}+
\Big[\min\left\{\sigma_u^2+\tau_u^2,\,
\sigma_u^2+ C{\widetilde{\mathscr{T}}},\,
\tau_u^2+ C{\widetilde{\mathscr{S}}},\,
C({\widetilde{\mathscr{T}}}+{\widetilde{\mathscr{S}}})\right\}
\chi_{B(x_0,R-\rho)\times [t_0-T+\delta,t_0]}\\
&\qquad\qquad+\left(\sigma_u^2 +\min\left\{\tau_u^2,\,C{\widetilde{\mathscr{T}}}\right\}\right)\chi_{(B(x_0,R)\setminus B(x_0,R-\rho))\times [t_0-T+\delta,t_0]}\\
&\qquad\qquad+
\left(\tau_u^2 +\min\left\{\sigma_u^2,\,C{\widetilde{\mathscr{S}}}\right\}\right)
\chi_{B(x_0,R-\rho)\times [t_0-T,t_0-T+\delta]}\\
&\qquad\qquad+\left(\sigma_u^2 +\tau_u^2\right)\chi_{(B(x_0,R)\setminus B(x_0,R-\rho))\times [t_0-T,t_0-T+\delta]}
\Big].\end{split}
\end{equation*}
Taking the square root and recalling~\eqref{TUTTITE-0}, we obtain~\eqref{B}, as desired.
\end{proof}

\section{Applications of Theorem~\ref{TH2}}\label{KAseecdt:3}
\setcounter{equation}{0}
\vskip2mm
\noindent
In this part, we will show several applications of Theorem~\ref{TH2} also
by comparing our general approach with some of the existing
results in specific situations.

\subsection{The heat equation in~$\R^n$}

As a special case,
one obtains from Theorem~\ref{TH2}
a global estimate for the gradient of solutions of the
heat equation in Euclidean balls. We state this byproduct
of Theorem~\ref{TH2} in detail for the sake of clarity:

\begin{cor}\label{CO}
Let~$B(x_0,R)\subset\R^n$ be the $n$-dimensional
Euclidean ball.
Let~$M>0$, $t_0\in\R$ and~$T>0$.
Let also~$Q_{R,T}:=B(x_0,R)\times[t_0-T,t_0]$ and suppose that~$u:Q_{R,T}\to
(0,M]$ is
a solution of
\begin{equation*}
u_t=\Delta u\qquad{\mbox{ in }}Q_{R,T}.
\end{equation*}
Then, for any $\delta\in(0,T)$ and $\rho\in(0,R)$,
there exists $C>0$, only depending on~$n$, such that
\begin{equation}\label{J34}
\aligned
&~~~~\frac{|\nabla u(x,t)|}{u(x,t)} \leq
\mathscr{Z}^{(0)}(x,t)\;
\left(1+\ln \frac{M}{u(x,t)}\right)\qquad{\mbox{ for all }}
(x,t)\in Q_{R,T},
\endaligned
\end{equation}
where
\begin{eqnarray*}
\mathscr{Z}^{(0)}:=
\beta_1^{(0)}\,{\mathscr{B}}_1
+\beta_2\,{\mathscr{B}}_2
+\beta_3\,{\mathscr{B}}_3
+\iota^{(0)}\,{\mathscr{I}},
\end{eqnarray*}
with
\begin{equation*}
\begin{split}
&\beta_1^{(0)}:=
\tau_u +\min\left\{\sigma_u,\,C{{\mathscr{S}}}^{(0)}\right\}\\{\mbox{and }}\quad
&\iota^{(0)}:=
\min\left\{\sigma_u+\tau_u,\,
\sigma_u+C{{\mathscr{T}} },\,
\tau_u+C{{\mathscr{S}}}^{(0)} ,\,
C({{\mathscr{T}}}+{{\mathscr{S}}}^{(0)})\right\}
,\end{split}
\end{equation*}
being
$$ {\mathscr{S}}^{(0)}:=\frac{1}{\rho}+\frac{1}{\sqrt{\rho(R-\rho)}}.$$
We recall that~${\mathscr{T}}$ is as in~\eqref{TUTTITE-0},
$\beta_2$ and $\beta_3$ are as in~\eqref{162}, $\tau_u$
and~$\sigma_u$ are as in~\eqref{sigtau}, and~${\mathscr{B}}_1$, ${\mathscr{B}}_2$, ${\mathscr{B}}_3$ and~${\mathscr{I}}$
are as in~\eqref{161}.
\end{cor}

\begin{proof} Corollary~\ref{CO} follows
directly from
Theorem~\ref{TH2}
by recalling~\eqref{161} and~\eqref{162}, since here~$k=0$ and $S(x,t,u)=0$.\end{proof}

We remark that the logarithmic function in \eqref{J34}
arises naturally in the context of heat equation: for instance,
one can consider the Gau{\ss} Kernel
$$ u_G(x,t):={\frac {1}{(4\pi t)^{n/2}}}e^{-|x|^{2}/4t}$$
and observe that, if~$x\in B\big((1,0,\dots,0),1/2\big)$
and~$t\in(1,2)$,
\begin{eqnarray*}
\frac{|\nabla u_G(x,t)|}{u_G(x,t)}=
\frac{|x|}{2t}=
-\frac{2}{|x|}\ln\left((4\pi t)^{n/2}\,u_G(x,t)\right)\simeq
\ln\frac{1}{u_G(x,t)}.
\end{eqnarray*}

\subsection{The heat equation on manifolds}\label{ECBD-90SINSCIUSTU}

A small variation of Corollary~\ref{CO}
provides the following result:

\begin{cor}\label{PIVOCO}
Let~$\mathscr{M}$ be a Riemannian manifold of dimension $n\geq2$,
with Ricci curvature bounded from below by some $k\in\mathbb{R}$.
Let~$B(x_0,R)$ be a geodesic ball in~$\mathscr{M}$
and~$Q_{R,T}:=B(x_0,R)\times[t_0-T,t_0]$.
Let~$M>0$, $t_0\in\R$ and~$T>0$.
Let~$u:Q_{R,T}\to
(0,M]$ be
a solution of
\begin{equation*}
u_t=\Delta u\qquad{\mbox{ in }}Q_{R,T}.
\end{equation*}
Then, for any $\delta\in(0,T)$ and $\rho\in(0,R)$,
there exists $C>0$, only depending on~$n$, such that
\begin{equation}\label{rFqchuevanedla}
\frac{|\nabla u(x,t)|}{u(x,t)} \leq \left(
C\,\sqrt{k_+}
+
\mathscr{Z}(x,t)\right)\;
\left(1+\ln \frac{M}{u(x,t)}\right)\qquad{\mbox{ for all }}
(x,t)\in Q_{R,T},
\end{equation}
where~$\mathscr{Z}$ is as in~\eqref{DEFdiZ}.
\end{cor}

\begin{proof}
We can exploit Theorem~\ref{TH2}
by suitably modifying~\eqref{161} and~\eqref{162}, since here $S(x,t,u)=0$,
and accordingly~$\mu=k_{+}$ and
${\mathscr{C}}=\sqrt{k_{+}}$.
\end{proof}

As particular cases of Corollary~\ref{PIVOCO}, one can re-obtain
several classical local estimates for the heat equation. We provide
one classical application to
show the comprehensive nature of the results provided in this paper.

\begin{theorem*}[Theorem 1.1 in~\cite{3SZ}] Let~$\mathscr{M}$ be a Riemannian manifold of dimension~$n\ge2$
with Ricci curvature bounded from below by~$-k$, for some~$k\ge 0$. Suppose that $u$ is any positive solution
to the heat equation in $Q_{R,T}:= B(x_0,R) \times[t_0-T,t_0]$. Suppose also that~$u\le M$ in~$ Q_{R,T}$. Then there exists a dimensional constant $C$ such that
\begin{equation}\label{KAn873:023838576}\frac{|\nabla u(x,t)|}{u(x,t)}\le C\,\left(\frac1R+\frac1{\sqrt{T}}+\sqrt{k}\right)\left(1+\ln\frac{M}{u(x,t)}\right)\end{equation}
for each~$(x,t)\in Q_{R/2,T/2}$.
\end{theorem*}

\begin{proof} We exploit Corollary~\ref{PIVOCO}
with~$\rho:=R/2$
and~$\delta:=T/2$. In this way, in view of~\eqref{161}
\begin{equation}\label{DENT9234}
\begin{split}&Q_{R/2,T/2}= B(x_0,R/2)\times [t_{0}-T/2,t_{0}] =B(x_0,R-\rho)\times[t_0-T+\delta,t_0]
\\&\qquad\subseteq
\{ {\mathscr{B}}_1=
{\mathscr{B}}_2
={\mathscr{B}}_3=0\}.\end{split}\end{equation}
Then, from~\eqref{sigtau}
and~\eqref{rFqchuevanedla},
for all~$(x,t)\in Q_{R/2,T/2}$,
\begin{equation}\label{J:JGAS:janzJA:jasw8e}
\begin{split}\frac{|\nabla u(x,t)|}{u(x,t)}\,& \leq \left(
C\,\sqrt{k}
+
\mathscr{Z}(x,t)\right)\;
\left(1+\ln \frac{M}{u(x,t)}\right)
\\&\le
\left(
C\,\sqrt{k}
+
\iota\right)\;
\left(1+\ln \frac{M}{u(x,t)}\right)
\\&=
\left(
C\,\sqrt{k}
+
\min\left\{\sigma_u+\tau_u,\,
\sigma_u+C{{\mathscr{T}} },\,
\tau_u+C{{\mathscr{S}}} ,\,
C({{\mathscr{T}}}+{{\mathscr{S}}})\right\}
\right)\\&\qquad\quad\cdot
\left(1+\ln \frac{M}{u(x,t)}\right).
\end{split}
\end{equation}
In particular,
by~\eqref{TUTTITE-0},
\begin{equation*}
\frac{|\nabla u|}{u}\le C\,
\left(\sqrt{k}
+{{\mathscr{T}}}+{{\mathscr{S}}}\right)
\left(1+\ln \frac{M}{u}\right)\le
C\,\left(\sqrt{k}+
\frac{1}{\sqrt{T}}+\frac{1}{R}+
\frac{\sqrt[4]{k}}{\sqrt{R}}\right)\left(1+\ln \frac{M}{u}\right).
\end{equation*}
Accordingly, since
$$ \frac{2\sqrt[4]{k}}{\sqrt{R}}
\le\sqrt{k}+
\frac{1}{{R}}
,$$
we see that
\begin{equation*}
\frac{|\nabla u|}{u}\le
C\,\left(
\sqrt{k}+
\frac{1}{\sqrt{T}}+
\frac{1}{R}\right)\left(1+\ln \frac{M}{u}\right),
\end{equation*}
which implies~\eqref{KAn873:023838576},
as desired.
\end{proof}

It is interesting to remark that the classical dependence on the
Ricci curvature in~\eqref{KAn873:023838576} is optimal,
as shown by the following explicit example. 
Let~$\lambda>0$ and~$D_\lambda$ be the unit disk in the plane with the Poincar\'e metrics
\begin{equation}\label{METRIX-1}
g_{ij}=\frac{4\lambda^2\,\delta_{ij}}{
(1-|x|^{2})^{2}}.\end{equation} We recall that in this setting the notion of harmonicity is conformally invariant, hence\footnote{For completeness we will provide a proof of~\eqref{METRIX-2}
in the appendix.}
\begin{equation}\label{METRIX-2}\begin{split}&
{\mbox{each harmonic function
in the Euclidean unit disk}}\\&{\mbox{is also harmonic in the hyperbolic metric of this Poincar\'e disk,}}\end{split}\end{equation} which in turn has negative
Ricci curvature of
order~$-\frac1{\lambda^2}$. Thus, we take~$u(x,t):=x_1+2$, we observe that
the supremum of~$u$ in the disk is equal to~$3$ and the infimum to~$1$,
and that~$u$ is a Euclidean, hence hyperbolic, harmonic function.
In particular, $u$ is a global solution of the heat equation
(hence we can consider the
case~$|x|\to+1$, which means~$R\to+\infty$ in the hyperbolic disk, and~$T\to+\infty$ in~$Q_{R,T}$).
Now, recalling (see page~20 in~\cite{MR1480173}) that~$\nabla u=g^{ij} u_i\partial_j$, we find that, in~$\{|x|<1/2\}$,
$$ |\nabla u|^2=g_{kj}g^{ij}g^{mk} u_i u_m=g^{mi} u_i u_m
=\frac{(1-|x|^2)^2}{4\lambda^2}u_i u_i=
\frac{(1-|x|^2)^2}{4\lambda^2}\ge\frac{9}{64\lambda^2}.$$
As a result, in~$\{|x|<1/2\}$,
$$ \frac{|\nabla u|}{u}\ge\frac{1}{8\lambda},$$and the latter quantity is of the order of the square root of minus the
Ricci curvature of~$D_\lambda$, thus
showing the optimality of~\eqref{KAn873:023838576}
with respect to the Ricci curvature (even in the
case of harmonic functions, i.e. stationary solutions of the heat equation).\medskip

One of the remarkable aspects of the estimates
in~\cite{3SZ} is their ``universality'' with respect to the parabolic boundary
data. On the other hand, the approach that we proposed with
cut-off functions improve the classical estimates when the parabolic
boundary data are ``exceptionally good''.
For instance, suppose that the function~$\frac{|\nabla u|}{u(1-v)}$
is controlled by a small~$\epsilon$ along the parabolic boundary
(hence, in view of~\eqref{sigtau}, both~$\tau_u$
and~$\sigma_u$ are bounded by~$\epsilon$): then one can deduce
from~\eqref{J:JGAS:janzJA:jasw8e} that
\begin{equation}\label{LAME} \frac{|\nabla u(x,t)|}{u(x,t)}\leq
C\left(
\sqrt{k}
+\epsilon\right)
\left(1+\ln \frac{M}{u(x,t)}\right),\end{equation}
which is an improvement of~\eqref{KAn873:023838576}
when~$\epsilon$ is small enough.
Interestingly, this improvement also occurs in the Euclidean case,
for the classical heat equation, in which~$k=0$.

Let us emphasize the fact that cases of this type take place even in very simple and explicit examples:
for instance, given~$\epsilon\in(0,1)$, one can consider the Euclidean case in which
$$ u(x,t)=10+\epsilon\,e^{x_1+t}.$$
In this situation, we have that, if~$|x|<1$ and~$t\in[0,1]$,
\begin{eqnarray*}&&
\Delta u=u_{11}=\epsilon\,e^{x_1+t}=u_t,\\
&&u\le 10+e^2\epsilon\le10+e^2< 19,\\
&&u\ge 10-e^2\epsilon\ge 10-e^2>1,\\
&&|\nabla u|=\epsilon\,e^{x_1+t}\le e^2\epsilon,\\
&&1-v=1+\ln\frac{M}u\in\left[ 1, 1+\ln19\right].
\end{eqnarray*}
Notice in particular that~$u$ solves the heat equation in~$Q_{1,1}$ and the classical estimate in~\eqref{KAn873:023838576}
(used here with~$k=0$, $R=t_0=T=1$) entails that
$$ \frac{|\nabla u|}{u}\le C\qquad{\mbox{in }}Q_{1/2,1/2}.$$
Instead, one can deduce from~\eqref{LAME} that
$$ \frac{|\nabla u|}{u}\leq
C\epsilon\qquad{\mbox{in }}Q_{1/2,1/2},$$
thus clarifying how the techniques developed in this paper
lead to an enhancement of the classical estimates
even when we are interested only in interior estimates, since they are capable of ``trading universality
for boundary information'' (namely, they can possibly take additional advantage of nice
boundary data whenever this information can lead to improved estimates with respect to the classical, universal ones).

Interestingly, this improvement also highlights
the persisting effect of ``exceptionally good'' parabolic boundary data
on the interior behavior of the solutions, and
in general it captures the boundary behavior of the solutions.

\subsection{General nonlinearities}

In this part, we will focus on the nonlinear
parabolic equation, by comparing our results with the previous literature.
A first improvement of our general result in Theorem~\ref{TH2}
is that it captures the global (and not only the local behavior)
of the solution.
Furthermore, the result in Theorem~\ref{TH2}
deals with several nonlinearities ``at the same time''
and from it we can re-obtain easily, as particular cases, a number of different
results that are scattered in several works of the existing literature.\medskip

As a matter of fact, Theorem~\ref{TH2} here
is new even when considered as an interior estimate.
We state this particular case explicitly as follows:

\begin{cor}\label{SOED:spe1}
Suppose that $u$ is a solution of equation~\eqref{20}
satisfying~\eqref{BOUND}. Then,
there exists~$C>0$, only depending on~$n$, such that
\begin{equation*}
\frac{|\nabla u(x,t)|}{u(x,t)} \leq \Big(C
\mathscr{C}+
\iota_\star\Big)\;
\left(1+\ln \frac{M}{u(x,t)}\right)\qquad{\mbox{ for all }}
(x,t)\in Q_{R/2,T/2}.
\end{equation*}
Here, we used the notation in~\eqref{TUTTITE-0} for~${\mathscr{C}}$,
and
\begin{equation}\label{BGblaolvlwebla1q2} \iota_\star:=
\min\left\{\sigma_u+\tau_u,\,
\sigma_u+C{{\mathscr{T}} }_\star,\,
\tau_u+C{{\mathscr{S}}}_\star ,\,
C({{\mathscr{T}}}_\star+{{\mathscr{S}}}_\star)\right\},\end{equation}
where
\begin{equation}\label{BGblaolvlwebla1q}
{\mathscr{T}}_\star:=\frac{1}{\sqrt{T}}\qquad{\mbox{ and }}\qquad
{\mathscr{S}}_\star:=\frac{1}{R}
+\frac{\sqrt[4]{k_+}}{\sqrt{R}}.
\end{equation}
\end{cor}

\begin{proof} The claim follows from~\eqref{B},
\eqref{162}
and~\eqref{DENT9234}. Notice in particular that,
in this setting, we have~$\rho:=R/2$ and~$\delta:=T/2$,
which give in~\eqref{TUTTITE-0} that~${\mathscr{T}}\le C{\mathscr{T}}_\star$
and~${\mathscr{S}}
\le {\mathscr{S}}_\star$.
\end{proof}

As a special case of Corollary~\ref{SOED:spe1},
one obtains the following new uniform interior estimate:

\begin{cor}\label{ILCOPEIR89}
Suppose that $u$ is a solution of equation~\eqref{20}
satisfying~\eqref{BOUND}. Then,
there exists~$C>0$, only depending on~$n$, such that
\begin{equation*}
\frac{|\nabla u(x,t)|}{u(x,t)} \leq C\Big(
\mathscr{C}
+
{{\mathscr{T}}}_\star+{{\mathscr{S}}}_\star\Big)\;
\left(1+\ln \frac{M}{u(x,t)}\right)\qquad{\mbox{ for all }}
(x,t)\in Q_{R/2,T/2}.
\end{equation*}
Here, we used the notation in~\eqref{TUTTITE-0} for~${\mathscr{C}}$,
and the one in~\eqref{BGblaolvlwebla1q} for
${\mathscr{T}}_\star$ and~${\mathscr{S}}_\star$.
\end{cor}

\begin{proof} By~\eqref{BGblaolvlwebla1q2},
we know in particular that~$\iota_\star\le
C({{\mathscr{T}}}_\star+{{\mathscr{S}}}_\star)$, hence the desired result follows
from Corollary~\ref{SOED:spe1}.
\end{proof}

In this way, we re-obtain many results
in the literature as particular cases. We list a few of them for
the sake of completeness. We start with a result
related to the thin film equation.

\begin{theorem*}[Theorem 1.1 in \cite{16MZ}] Let~$\alpha$, $\lambda\in\R$.
Let~$\mathscr{M}$
be a Riemannian manifold of dimension~$n\ge2$
with Ricci curvature bounded from below by~$-k$,
where $k$ is a non-negative constant.
Suppose that $u$
is a positive solution to $$u_t=\Delta u+\lambda u^\alpha$$
in~$ Q_{R,T}:=B(x_0,R) \times[0, T]$.
Let~$M:=\displaystyle\sup_{Q_{R,T}}u$
and~$m:=\displaystyle\inf_{Q_{R,T}}u$.

Then\footnote{We think that
there are some typos in Theorem 1.1 in \cite{16MZ},
since the claim ``Then in $Q_{R,T}$''
should read ``in $Q_{R/2,T/2}$''. The necessity
of reducing the domain in~\cite{16MZ}
comes from formula (2.11) there.
In addition, it seems there could be some constants missing
in formula~(1.5), and also in formula~(1.6)
when~$\lambda>0$ and~$\alpha<0$
of~\cite{16MZ}, since one term has a negative
sign (these constants should probably have
appeared in formulas~(2.20) and~(2.26)
in~\cite{16MZ} and the proof
should take care of the delicate situation in which
the maximal point there occurs on small
values of the cut-off function).
To avoid confusion, we do not include
the unclear formulas in our version of
the main theorem of~\cite{16MZ}.
On the other hand, it seems to us that the cases~$\lambda=0$,
$\alpha=0$
and~$\alpha=1$, which were in principle
omitted in the original formulation of
Theorem 1.1 in \cite{16MZ}, can be included without
extra effort, hence these cases are
explicitly present in the formulation
given here.} in $Q_{R/2,T/2}$, we have
\begin{itemize}
\item if~$\lambda<0$ and~$\alpha\in(-\infty,1]$,
\begin{equation}\label{CAL:Z3}
\frac{|\nabla u|^2}{u^2}\le
C\,\left(k+\frac1{R^2}+\frac1T+\lambda(\alpha-1)m^{\alpha-1}\right)
\left(1+\ln\frac{M}{u}\right)^2;
\end{equation}
\item if~$\lambda\ge0$ and~$\alpha\in [1,+\infty)$,
\begin{equation}\label{CAL:Z2}
\frac{|\nabla u|^2}{u^2}\le
C\,\left(k+\frac1{R^2}+\frac1T+\lambda\alpha M^{\alpha-1}\right)
\left(1+\ln\frac{M}{u}\right)^2;\end{equation}
\item if~$\lambda\ge0$ and~$\alpha\in[0,1)$,
\begin{equation}\label{CAL:Z1}
\frac{|\nabla u|^2}{u^2}\le
C\,\left(k+\frac1{R^2}+\frac1T+\lambda\alpha m^{\alpha-1}\right)
\left(1+\ln\frac{M}{u}\right)^2.\end{equation}
\end{itemize}
Here, the constant C depends only on the dimension n.
\end{theorem*}

\begin{proof} We can use the setting in~\eqref{20}
with~$t_0:=T$ and~$S(x,t,u):=\lambda u^\alpha$. With this, recalling~\eqref{GAMMA}
and~\eqref{LA:23}, we see that
\begin{eqnarray*}
\gamma=0\qquad{\mbox{ and }}\qquad
\mu=\sup_{ Q_{R,T}}
\left(k+
(\alpha-1)\lambda u^{\alpha-1}+
\frac{ \lambda u^{\alpha-1}}{1-v}
\right)_+.\end{eqnarray*}
Consequently, by~\eqref{TUTTITE-0},
$${\mathscr{C}}=
\sup_{ Q_{R,T}}\sqrt{
\left(k+
(\alpha-1)\lambda u^{\alpha-1}+
\frac{ \lambda u^{\alpha-1}}{1-v}
\right)_+}.$$
Hence, in light of~\eqref{BGblaolvlwebla1q},
\begin{equation*}\begin{split}\mathscr{C}+
{{\mathscr{T}}}_\star+{{\mathscr{S}}}_\star \le
\sup_{ Q_{R,T}}\sqrt{
\left(k+
(\alpha-1)\lambda u^{\alpha-1}+
\frac{ \lambda u^{\alpha-1}}{1-v}
\right)_+}+
\frac{1}{\sqrt{T}}
+\frac{1}{R}
+\frac{\sqrt[4]{k}}{\sqrt{R}}
.\end{split}\end{equation*}
Then, we can exploit Corollary~\ref{ILCOPEIR89} in this
setting, which yields that, in~$Q_{R/2,T/2}$,
\begin{equation}\label{vgj-vgh-vb}
\begin{split}&\mathscr{E}:=
\left(1+\ln \frac{M}{u(x,t)}\right)^{-2}\;
\frac{|\nabla u|^2}{u^2} \leq C\Big(
\mathscr{C}+
{{\mathscr{T}}}_\star+{{\mathscr{S}}}_\star\Big)^2\\
&\qquad\le C\,\bigg[
\sup_{ Q_{R,T}}
\left(k+
(\alpha-1)\lambda u^{\alpha-1}+
\frac{ \lambda u^{\alpha-1}}{1-v}
\right)_++
\frac{1}{{T}}
+\frac{1}{R^2}
+\frac{\sqrt{k}}{{R}}
\bigg],
\end{split}\end{equation}
up to renaming~$C$ line after line.

Now, if~$\lambda\ge0$ and~$\alpha\in[0,1)$, recalling~\eqref{Jna:0},
$$ k+
(\alpha-1)\lambda u^{\alpha-1}+
\frac{ \lambda u^{\alpha-1}}{1-v}\le
k+(\alpha-1)\lambda u^{\alpha-1}+
\lambda u^{\alpha-1}\le
k+\lambda\alpha m^{\alpha-1}.
$$
{F}rom this and~\eqref{vgj-vgh-vb}, we find that
\begin{eqnarray*} \mathscr{E}&\le&
C\,\bigg[
\left(
k+\lambda\alpha m^{\alpha-1}
\right)+
\frac{1}{{T}}
+\frac{1}{R^2}
+\frac{\sqrt{k}}{{R}}
\bigg]\\&\le&
C\,\bigg(
k
+\lambda\alpha m^{\alpha-1}
+
\frac{1}{{T}}
+\frac{1}{R^2}
\bigg),
\end{eqnarray*}
from which~\eqref{CAL:Z1}
plainly follows.

If instead~$\lambda\ge0$ and~$\alpha\in[1,+\infty)$,
using that~$v\le0$, we remark that
$$ k+
(\alpha-1)\lambda u^{\alpha-1}+
\frac{ \lambda u^{\alpha-1}}{1-v}\le
k+
(\alpha-1)\lambda u^{\alpha-1}+
\lambda u^{\alpha-1}
=k+
\lambda \alpha u^{\alpha-1}\le
k+
\lambda\alpha M^{\alpha-1}.$$
This and~\eqref{vgj-vgh-vb} give that
\begin{equation*}
\begin{split}\mathscr{E}\,&\le
C\,\bigg[
\big(
k+
\lambda\alpha M^{\alpha-1}\big)+
\frac{1}{{T}}
+\frac{1}{R^2}
+\frac{\sqrt{k}}{{R}}
\bigg]\\&\le
C\,\bigg(
k
+
\lambda\alpha M^{\alpha-1}+
\frac{1}{{T}}
+\frac{1}{R^2}
\bigg),
\end{split}\end{equation*}
which
proves~\eqref{CAL:Z2}.

Now, we suppose that~$\lambda<0$
and~$\alpha\in(-\infty,1]$.
In this case we see that
$$ (\alpha-1)\lambda u^{\alpha-1}=
\frac{|(\alpha-1)\lambda|}{ u^{|1-\alpha|}}
\le \frac{|(\alpha-1)\lambda|}{ m^{|1-\alpha|}}=
(\alpha-1)\lambda m^{\alpha-1}.
$$
{F}rom this, one deduces that
$$ k+
(\alpha-1)\lambda u^{\alpha-1}+
\frac{ \lambda u^{\alpha-1}}{1-v}\le
k+
(\alpha-1)\lambda u^{\alpha-1}\le k+
(\alpha-1)\lambda m^{\alpha-1}.$$
Therefore, one can use this information and~\eqref{vgj-vgh-vb}
to obtain~\eqref{CAL:Z3}, as desired.
\end{proof}

{F}rom our main results a general
gradient estimate for solutions
of semilinear parabolic equations
follows as a byproduct. For concreteness,
we point out the following explicit result:

\begin{cor}\label{MANTCOR}
Let~$p\in\R$.
Let~$\mathscr{M}$
be a Riemannian manifold with~$
\mathrm{Ric}(\mathscr{M})\geq-k$ for some~$k\in\R$.
Let~$ u$ be a positive solution to the
semilinear heat equation
\begin{equation}\label{8j7yh6tg4ed9o2ewd0so} u_t=\Delta u+u^p\end{equation}
in~$Q_{R,T}$. Assume that $u\le M$ in~$Q_{R,T}$.
Then, there exists~$C>0$ depending on~$n$ such that,
on~$ Q_{R/2,T/2}$, there holds
$$\frac{|\nabla u(x,t)|}{u(x,t)}\le
C\left(\max\left\{
\sqrt{k_+}\,,\sqrt{\left(k+p\,\vartheta^{p-1}
\right)_+}\right\}
+\frac{1}{\sqrt{T}}+
\frac{1}{R}
\right)
\left(1+\ln \frac{M}{u(x,t)}\right),$$
where
\begin{equation}\label{chietetea} \vartheta:=\begin{cases}
M &{\mbox{ if }} p>1,\\
1&{\mbox{ if }} p=1,\\
\displaystyle
\inf_{Q_{R,T}} u&{\mbox{ if }} p\in(0,1),\\
0&{\mbox{ if }} p=0,\\
M
&{\mbox{ if }} p<0.
\end{cases}\end{equation}
\end{cor}

\begin{proof}
We take~$p\in\R$ and~$S(x,t,u):=u^p$.
In this way, the notations in~\eqref{GAMMA}
and~\eqref{LA:23} yield that
\begin{equation*}
\begin{split}
\gamma\,&=0\\ {\mbox{and}}\qquad
\mu\,&=\sup_{ (x,t)\in Q_{R,T}}
\left(k+
pu^{p-1}
-u^{p-1}+\frac{ u^{p-1}}{1-v}
\right)_+\\
&=\sup_{ (x,t)\in Q_{R,T}}
\left(k+
\left((p-1)+\frac{ 1}{1-v}\right)u^{p-1}
\right)_+\\
&\le\sup_{ (x,t)\in Q_{R,T}}
\left(k+pu^{p-1}
\right)_+\\
&\le
\left(k+p\,\vartheta^{p-1}
\right)_+
.\end{split}\end{equation*}
Hence, by~\eqref{TUTTITE-0},
\begin{equation*}
{\mathscr{C}}\le
\sqrt{\left(k+p\,\vartheta^{p-1}
\right)_+}.\end{equation*}
Using this, \eqref{BGblaolvlwebla1q}
and Corollary~\ref{ILCOPEIR89},
we thereby conclude that, in~$Q_{R/2,T/2}$,
\begin{equation}\label{Pijwdyrrgmalr340}
\begin{split}
\left(1+\ln \frac{M}{u(x,t)}\right)^{-1}
\frac{|\nabla u(x,t)|}{u(x,t)}\leq\,& C
\big(\mathscr{C}+
{{\mathscr{T}}}_\star+{{\mathscr{S}}}_\star\big)\\
\le\,& C\left(
\sqrt{\left(k+p\,\vartheta^{p-1}
\right)_+}
+\frac{1}{\sqrt{T}}+
\frac{1}{R}
+\frac{\sqrt[4]{k_+}}{\sqrt{R}}
\right).
\end{split}
\end{equation}
We also remark that, by the Cauchy-Schwarz inequality,
$$ \frac{\sqrt[4]{k_+}}{2\sqrt{R}}
\le
\sqrt{k_+}
+
\frac{1}{R}.$$
This and~\eqref{Pijwdyrrgmalr340}
give that
\begin{equation*}
\left(1+\ln \frac{M}{u(x,t)}\right)^{-1}
\frac{|\nabla u(x,t)|}{u(x,t)}\leq C\left(
\sqrt{k_+}\,+\,\sqrt{\left(k+p\,\vartheta^{p-1}
\right)_+}
+\frac{1}{\sqrt{T}}+
\frac{1}{R}
\right).
\end{equation*}
This yields the desired result.
\end{proof}

We remark that
Corollary~\ref{MANTCOR}
contains, as a special case, a recent result obtained in~\cite{12345}
which dealt with the case~$p>1$ (see in particular Lemma~3.1
in~\cite{12345}).\medskip

Moreover, we think that an interesting treat of our Corollary~\ref{MANTCOR}
in its general formulation
is that the constant~$C$ is independent of~$p$:
besides its technical relevance, this fact reveals a telling
feature of the nonlinear parabolic equations,
in the sense that, at a formal level, for large~$p$,
given~$a\in(0,1)$,
solutions~$u=u_p$ of~\eqref{8j7yh6tg4ed9o2ewd0so}
with~$0<u\le1-a$
satisfy, on~$ Q_{R/2,T/2}$,
\[\frac{|\nabla u|}{u}\le
C\left(\max\left\{
\sqrt{k_+}\,,\sqrt{\left(k+p\,(1-a)^{p-1}
\right)_+}\right\}
+\frac{1}{\sqrt{T}}+
\frac{1}{R}
\right)
\left(1+|\ln u|\right),\]
which, as $p\to+\infty$, formally
boils down to
\[\frac{|\nabla u|}{u}\le
C\left(
\sqrt{k_+}
+\frac{1}{\sqrt{T}}+
\frac{1}{R}
\right)
\left(1+|\ln u|\right),\]
which recovers the estimate for the heat equation given
in~\eqref{KAn873:023838576} (we remark that also equation~\eqref{8j7yh6tg4ed9o2ewd0so}
reduces to the heat equation as~$p\to+\infty$ in this regime,
and that the assumption~$u\le1-a$ is equivalent
to~$u$ bounded in the case of the heat equation
due to its linear structure). Though we do not address a rigorous
treatment of these limit
properties as~$p\to+\infty$
here, we think that our unified approach to gradient estimates entails a number of interesting connections between structurally different equations which could be worth a further exploration.
\medskip

In addition, from Corollary~\ref{MANTCOR},
one re-obtains a recent result
motivated by ancient solutions:

\begin{theorem*}[Lemma\footnote{See also
the enhanced version of~\cite{4CM}
available on {\tt http://cvgmt.sns.it/paper/3135/}} 4.1 in~\cite{4CM}]
Let~$\mathscr{M}$
be a Riemannian manifold with~$
\mathrm{Ric}(\mathscr{M})\geq-k$ for some~$k\in\R$.
Let~$ u$ be a positive solution to the
semilinear heat equation
$$ u_t=\Delta u+u^2$$
in~$Q_{R,T}$. Assume that $u\le M$ in~$Q_{R,T}$.
Then, there exists~$C>0$ depending on~$n$ such that,
on~$ Q_{R/2,T/2}$, there holds
\begin{equation}\label{MANUOVA}
\frac{|\nabla u|}{u}\le
C\,\left(\frac1R+\frac1{\sqrt{T}}+\sqrt{(2M+k)_+}\right)
\left(1+\ln\frac{M}{u}\right).
\end{equation}
\end{theorem*}

\begin{proof} Recalling~\eqref{chietetea},
we see that,
when~$p=2$,
\begin{equation*}
\sqrt{k_+}\le
\sqrt{\left(k+2M\right)_+}
=\sqrt{\left(k+p\,\vartheta^{p-1}
\right)_+}.
\end{equation*}
{F}rom this and Corollary~\ref{MANTCOR},
we obtain~\eqref{MANUOVA}.
\end{proof}

\section*{Acknowledgments}

\noindent
Cecilia Cavaterra has been partially supported by GNAMPA 
(Gruppo Nazionale per l'Analisi Matematica, la Probabilit\`{a}
e le loro Applicazioni) of INdAM (Istituto Nazionale di Alta Matematica).
Serena Dipierro and Enrico Valdinoci are members of INdAM and AustMS.
Serena Dipierro has been supported by the Australian Research Council DECRA
DE180100957 ``PDEs, free boundaries and applications''.
Enrico Valdinoci has been supported by the Australian Laureate Fellowship FL190100081
``Minimal surfaces, free boundaries and partial differential equations''.
Zu Gao has been supported by the Independent Innovation Research Fund of Wuhan University of Technology (No: 2021IVA058).

\begin{appendix}
\section{Proof of~\eqref{METRIX-2}}

We recall that a metric~$g_{ij}$ is said to be conformal
(or, more precisely, conformal to the Euclidean metric) if
\begin{equation}\label{CONFOG}
g_{ij}=\varphi \delta_{ij},
\end{equation}
for some scalar factor~$\varphi$. For instance, the metric of the
Poincar\'e disk in the plane given in~\eqref{METRIX-1} is conformal, with factor~$\varphi:=
\frac{4\lambda^2}{
(1-|x|^{2})^{2}}$, being~$|\cdot|$ the standard Euclidean norm.

The Laplacian operator (or, more precisely, the Laplace-Beltrami operator)
possesses an explicit representation with respect to conformal metrics: roughly speaking,
since conformal metrics preserve angles, an infinitesimal orthonormal frame
is transformed into an infinitesimal orthogonal frame (the length of the vectors possibly being affected
by the conformal factor~$\varphi$), thus the new Laplacian (being computed as sum of second derivatives
with respect to an orthonormal frame) remains the same possibly up
to a ``curvature'' term which accounts for the variation of~$\varphi$ (this additional term pops up because the Laplacian
is a second order operator). The case of dimension~$2$ is somewhat special, since this additional
term vanishes.

Here are the explicit computations underpinning this heuristic idea. {F}rom~\eqref{CONFOG}, we have that~$g^{ij}=\varphi^{-1}\delta^{ij}$ and~$\det g=\varphi^{n}$. Hence,
in local coordinates, the Laplacian with respect to the conformal metrics in~\eqref{CONFOG} is
\begin{eqnarray*} &&\frac{1}{\sqrt{\det g}} \sum_{i,j=1}^n \partial_i\Big( \sqrt{\det g} \;g^{ij} \partial_j\Big)=
\frac{1}{\varphi^{\frac{n}2}} \sum_{i,j=1}^n \partial_i\Big( \varphi^{\frac{n-2}2} \,\delta^{ij} \partial_j\Big)=
\frac{1}{\varphi^{\frac{n}2}} \sum_{i=1}^n \partial_i\Big( \varphi^{\frac{n-2}2}  \partial_i\Big)\\&&\qquad\qquad=
\frac{1}{\varphi^{\frac{n}2}} \sum_{i=1}^n \left( \frac{n-2}2\,\varphi^{\frac{n-4}2} \partial_i\varphi\, \partial_i
+\varphi^{\frac{n-2}2}  \partial_{ii}\right)
=\sum_{i=1}^n \left( \frac{n-2}{2\,\varphi^2}  \partial_i\varphi\,\partial_i
+\frac{1}{\varphi}  \partial_{ii}\right).
\end{eqnarray*}
In dimension~$2$, this boils down to
$$ \frac{1}{\varphi} \,\sum_{i=1}^n \partial_{ii},$$
which is a scalar multiple of the Euclidean Laplacian, and therefore~\eqref{METRIX-2} plainly follows.

\end{appendix}


\vskip6mm
\begin{bibdiv}
\begin{biblist}

\bib{MR3450752}{article}{
   author={Attouchi, Amal},
   title={Gradient estimate and a Liouville theorem for a $p$-Laplacian
   evolution equation with a gradient nonlinearity},
   journal={Differential Integral Equations},
   volume={29},
   date={2016},
   number={1-2},
   pages={137--150},
   issn={0893-4983},
   review={\MR{3450752}},
}

\bib{BERNOI}{article}{
   author={Cabr\'{e}, Xavier},
   author={Dipierro, Serena},
   author={Valdinoci, Enrico},
      title={The Bernstein technique for integro-differential equations},
            journal = {arXiv e-prints},
      date={2020},
      eprint={2010.00376},
      archivePrefix={arXiv},
      primaryClass={math.AP}
}

\bib{MR1351007}{book}{
   author={Caffarelli, Luis A.},
   author={Cabr\'{e}, Xavier},
   title={Fully nonlinear elliptic equations},
   series={American Mathematical Society Colloquium Publications},
   volume={43},
   publisher={American Mathematical Society, Providence, RI},
   date={1995},
   pages={vi+104},
   isbn={0-8218-0437-5},
   review={\MR{1351007}},
   doi={10.1090/coll/043},
}

\bib{MR1296785}{article}{
   author={Caffarelli, Luis},
   author={Garofalo, Nicola},
   author={Seg\`ala, Fausto},
   title={A gradient bound for entire solutions of quasi-linear equations
   and its consequences},
   journal={Comm. Pure Appl. Math.},
   volume={47},
   date={1994},
   number={11},
   pages={1457--1473},
   issn={0010-3640},
   review={\MR{1296785}},
   doi={10.1002/cpa.3160471103},
}

\bib{4CM}{article}{
   author={Castorina, Daniele},
   author={Mantegazza, Carlo},
   title={Ancient solutions of semilinear heat equations on Riemannian
   manifolds},
   journal={Atti Accad. Naz. Lincei Rend. Lincei Mat. Appl.},
   volume={28},
   date={2017},
   number={1},
   pages={85--101},
   issn={1120-6330},
   review={\MR{3621772}},
   doi={10.4171/RLM/753},
}

\bib{12345}{article}{
   author={Castorina, Daniele},
   author={Mantegazza, Carlo},
   title={Ancient solutions of superlinear heat equations on
Riemannian manifolds},
   journal={Commun. Contemp. Math.},
   date={to appear},
}

\bib{2019arXiv190304569C}{article}{
 author={Cavaterra, Cecilia},
   author={Dipierro, Serena},
   author={Farina, Alberto},
   author={Gao, Zu},
   author={Valdinoci, Enrico},
   title={Pointwise gradient bounds for entire solutions of elliptic
   equations with non-standard growth conditions and general nonlinearities},
   journal={J. Differential Equations},
   volume={270},
   date={2021},
   pages={435--475},
   issn={0022-0396},
   review={\MR{4150380}},
   doi={10.1016/j.jde.2020.08.007},
}

\bib{17CZ}{article}{
   author={Chen, Qun},
     author={Zhao, Guangwen},
    title={Li-Yau type and Souplet-Zhang type gradient estimates of a parabolic equation for the V-Laplacian},
   journal={J. Math. Anal. Appl.},
   volume={463},
   date={2018},
   number={2},
   pages={744--759},
   issn={0022-247X},
   review={\MR{3785481}},
   doi={10.1016/j.jmaa.2018.03.049},
}

\bib{MR385749}{article}{
   author={Cheng, S. Y.},
   author={Yau, S. T.},
   title={Differential equations on Riemannian manifolds and their geometric
   applications},
   journal={Comm. Pure Appl. Math.},
   volume={28},
   date={1975},
   number={3},
   pages={333--354},
   issn={0010-3640},
   review={\MR{385749}},
   doi={10.1002/cpa.3160280303},
}

\bib{MR3231999}{article}{
   author={Cozzi, Matteo},
   author={Farina, Alberto},
   author={Valdinoci, Enrico},
   title={Gradient bounds and rigidity results for singular, degenerate,
   anisotropic partial differential equations},
   journal={Comm. Math. Phys.},
   volume={331},
   date={2014},
   number={1},
   pages={189--214},
   issn={0010-3616},
   review={\MR{3231999}},
   doi={10.1007/s00220-014-2107-9},
}

\bib{MR4238774}{article}{
   author={Dipierro, Serena},
   author={Gao, Zu},
   author={Valdinoci, Enrico},
   title={Global gradient estimates for nonlinear parabolic operators},
   journal={ESAIM Control Optim. Calc. Var.},
   volume={27},
   date={2021},
   pages={Paper No. 21, 37},
   issn={1292-8119},
   review={\MR{4238774}},
   doi={10.1051/cocv/2021016},
}

\bib{14DK}{article}{
   author={Dung, Nguyen Thac},
     author={Khanh, Nguyen Ngoc},
    title={Gradient estimates of Hamilton-Souplet-Zhang type for a general heat equation on Riemannian manifolds},
   journal={Arch. Math. (Basel)},
   volume={105},
   date={2015},
   number={5},
   pages={479--490},
   issn={0003-889X},
   review={\MR{3413923}},
   doi={10.1007/s00013-015-0828-4},
}

\bib{14DKN}{article}{
   author={Dung, Nguyen Thac},
   author={Khanh, Nguyen Ngoc},
   author={Ng\^{o}, Qu\^{o}c Anh},
   title={Gradient estimates for some $f$-heat equations driven by
   Lichnerowicz's equation on complete smooth metric measure spaces},
   journal={Manuscripta Math.},
   volume={155},
   date={2018},
   number={3-4},
   pages={471--501},
   issn={0025-2611},
   review={\MR{3763415}},
   doi={10.1007/s00229-017-0946-3},
}

\bib{MR4021092}{article}{
   author={Ha Tuan Dung},
   author={Nguyen Thac Dung},
   title={Sharp gradient estimates for a heat equation in Riemannian
   manifolds},
   journal={Proc. Amer. Math. Soc.},
   volume={147},
   date={2019},
   number={12},
   pages={5329--5338},
   issn={0002-9939},
   review={\MR{4021092}},
   doi={10.1090/proc/14645},
}

\bib{MR2680184}{article}{
   author={Farina, Alberto},
   author={Valdinoci, Enrico},
   title={A pointwise gradient estimate in possibly unbounded domains with
   nonnegative mean curvature},
   journal={Adv. Math.},
   volume={225},
   date={2010},
   number={5},
   pages={2808--2827},
   issn={0001-8708},
   review={\MR{2680184}},
   doi={10.1016/j.aim.2010.05.008},
}

\bib{MR2812957}{article}{
   author={Farina, Alberto},
   author={Valdinoci, Enrico},
   title={A pointwise gradient bound for elliptic equations on compact
   manifolds with nonnegative Ricci curvature},
   journal={Discrete Contin. Dyn. Syst.},
   volume={30},
   date={2011},
   number={4},
   pages={1139--1144},
   issn={1078-0947},
   review={\MR{2812957}},
   doi={10.3934/dcds.2011.30.1139},
}


\bib{MR1230276}{article}{
   author={Hamilton, Richard S.},
   title={A matrix Harnack estimate for the heat equation},
   journal={Comm. Anal. Geom.},
   volume={1},
   date={1993},
   number={1},
   pages={113--126},
   issn={1019-8385},
   review={\MR{1230276}},
   doi={10.4310/CAG.1993.v1.n1.a6},
}

\bib{13HM}{article}{
   author={Huang, Guangyue},
   author={Ma, Bingqing},
    title={Hamilton's gradient estimates of porous medium and fast diffusion equations},
   journal={Geom. Dedicata},
   volume={188},
   date={2017},
   number={},
   pages={1--16},
   issn={0046-5755},
   review={\MR{3639621}},
   doi={10.1007/s10711-016-0201-1},
}

\bib{5J}{article}{
   author={Jiang, Xinrong},
   title={Gradient estimate for a nonlinear heat equation on Riemannian
   manifolds},
   journal={Proc. Amer. Math. Soc.},
   volume={144},
   date={2016},
   number={8},
   pages={3635--3642},
   issn={0002-9939},
   review={\MR{3503732}},
   doi={10.1090/proc/12995},
}

\bib{MR0114050}{article}{
   author={Lady\v{z}enskaya, O. A.},
   title={Solution of the first boundary problem in the large for
   quasi-linear parabolic equations},
   language={Russian},
   journal={Trudy Moskov. Mat. Ob\v{s}\v{c}.},
   volume={7},
   date={1958},
   pages={149--177},
   issn={0134-8663},
   review={\MR{0114050}},
}

\bib{6LY}{article}{
   author={Li, Peter},
   author={Yau, Shing-Tung},
   title={On the parabolic kernel of the Schr\"{o}dinger operator},
   journal={Acta Math.},
   volume={156},
   date={1986},
   number={3-4},
   pages={153--201},
   issn={0001-5962},
   review={\MR{834612}},
   doi={10.1007/BF02399203},
}

\bib{16MZ}{article}{
   author={Ma, Bingqing},
     author={Zeng, Fanqi},
    title={Hamilton-Souplet-Zhang's gradient estimates and Liouville theorems for a nonlinear parabolic equation},
   journal={C. R. Math. Acad. Sci. Paris, Ser. I},
   volume={356},
   date={2018},
   number={5},
   pages={550--557},
   issn={1631-073X},
   review={\MR{3790428}},
   doi={10.1016/j.crma.2018.04.003},
}

\bib{7MZS}{article}{
   author={Ma, Li},
   author={Zhao, Lin},
   author={Song Xianfa},
   title={Gradient estimate for the degenerate parabolic
   equation $u_t=\Delta F(u)+H(u)$ on manifolds},
   journal={J. Differential Equations},
   volume={244},
   date={2008},
   number={5},
   pages={1157--1177},
   issn={0022-0396},
   review={\MR{2392508}},
   doi={10.1016/j.jde.2007.08.014},
}

\bib{MR803255}{article}{
   author={Modica, Luciano},
   title={A gradient bound and a Liouville theorem for nonlinear Poisson
   equations},
   journal={Comm. Pure Appl. Math.},
   volume={38},
   date={1985},
   number={5},
   pages={679--684},
   issn={0010-3640},
   review={\MR{803255}},
   doi={10.1002/cpa.3160380515},
}

\bib{MR2327126}{book}{
   author={Oprea, John},
   title={Differential geometry and its applications},
   series={Classroom Resource Materials Series},
   edition={2},
   publisher={Mathematical Association of America, Washington, DC},
   date={2007},
   pages={xxii+469},
   isbn={978-0-88385-748-9},
   review={\MR{2327126}},
}

\bib{MR454338}{article}{
   author={Payne, L. E.},
   title={Some remarks on maximum principles},
   journal={J. Analyse Math.},
   volume={30},
   date={1976},
   pages={421--433},
   issn={0021-7670},
   review={\MR{454338}},
   doi={10.1007/BF02786729},
}

\bib{MR1480173}{book}{
   author={Petersen, Peter},
   title={Riemannian geometry},
   series={Graduate Texts in Mathematics},
   volume={171},
   publisher={Springer-Verlag, New York},
   date={1998},
   pages={xvi+432},
   isbn={0-387-98212-4},
   review={\MR{1480173}},
   doi={10.1007/978-1-4757-6434-5},
}

\bib{MR0402274}{article}{
   author={Serrin, James},
   title={Gradient estimates for solutions of nonlinear elliptic and
   parabolic equations},
   conference={
      title={Contributions to nonlinear functional analysis},
      address={Proc. Sympos., Math. Res. Center, Univ. Wisconsin, Madison,
      Wis.},
      date={1971},
   },
   book={
      publisher={Academic Press, New York},
   },
   date={1971},
   pages={565--601},
   review={\MR{0402274}},
}

\bib{SIRALORO}{article}{
author = {Sirakov, Boyan},
author ={Souplet, Philippe},
        title = {Liouville-type theorems for unbounded solutions of elliptic equations in half-spaces},
      journal = {arXiv e-prints},
         date = {2020},
                eprint = {2002.07247},
archivePrefix = {arXiv},
 primaryClass = {math.AP},
}

\bib{3SZ}{article}{
   author={Souplet, Philippe},
   author={Zhang, Qi S.},
   title={Sharp gradient estimate and Yau's Liouville theorem for the heat
   equation on noncompact manifolds},
   journal={Bull. London Math. Soc.},
   volume={38},
   date={2006},
   number={6},
   pages={1045--1053},
   issn={0024-6093},
   review={\MR{2285258}},
   doi={10.1112/S0024609306018947},
}

\bib{MR615561}{book}{
   author={Sperb, Ren\'{e} P.},
   title={Maximum principles and their applications},
   series={Mathematics in Science and Engineering},
   volume={157},
   publisher={Academic Press, Inc. [Harcourt Brace Jovanovich, Publishers],
   New York-London},
   date={1981},
   pages={ix+224},
   isbn={0-12-656880-4},
   review={\MR{615561}},
}

\bib{MR3254790}{article}{
   author={Teixeira, Eduardo V.},
   author={Urbano, Jos\'{e} Miguel},
   title={An intrinsic Liouville theorem for degenerate parabolic equations},
   journal={Arch. Math. (Basel)},
   volume={102},
   date={2014},
   number={5},
   pages={483--487},
   issn={0003-889X},
   review={\MR{3254790}},
   doi={10.1007/s00013-014-0648-y},
}

\bib{8W}{article}{
   author={Wu, Jiayong},
    title={Gradient Estimates for a Nonlinear Diffusion Equation on Complete Manifolds},
   journal={J. Partial Differ. Equ.},
   volume={23},
   date={2010},
   number={1},
   pages={68--79},
   issn={1000--940X},
   review={\MR{2640448}},
   doi={10.4208/jpde.v23.n1.4},
}

\bib{15W}{article}{
   author={Wu, Jiayong},
    title={Elliptic gradient estimates for a weighted heat equation and applications},
   journal={Math. Z.},
   volume={280},
   date={2015},
   number={1-2},
   pages={451--468},
   issn={0025-5874},
   review={\MR{3343915}},
   doi={10.1007/s00209-015-1432-9},
}

\bib{11X}{article}{
   author={Xu, Xiangjin},
    title={Gradient estimates for $u_t=\Delta F(u)$ on manifolds and some Liouville-type theorems},
   journal={J. Differential Equations},
   volume={252},
   date={2012},
   number={2},
   pages={1403--1420},
   issn={0022-0396},
   review={\MR{2853544}},
   doi={10.1016/j.jde.2011.08.004},
}

\bib{9Z}{article}{
   author={Zhu, Xiaobao},
    title={Hamilton's gradient estimates and Liouville theorems for fast diffusion equations on noncompact Remannian manifolds},
   journal={Proc. Amer. Math. Soc.},
   volume={139},
   date={2011},
   number={5},
   pages={1637--1644},
   issn={0002-9939},
   review={\MR{2763753}},
   doi={10.1090/S0002-9939-2010-10824-9},
}

\bib{MR2810695}{article}{
   author={Zhu, Xiaobao},
   title={Gradient estimates and Liouville theorems for nonlinear parabolic
   equations on noncompact Riemannian manifolds},
   journal={Nonlinear Anal.},
   volume={74},
   date={2011},
   number={15},
   pages={5141--5146},
   issn={0362-546X},
   review={\MR{2810695}},
   doi={10.1016/j.na.2011.05.008},
}

\bib{10Z}{article}{
   author={Zhu, Xiaobao},
    title={Hamilton's gradient estimates and Liouville theorems for porous medium equations on noncompact Riemannian manifolds},
   journal={J. Math. Anal. Appl.},
   volume={402},
   date={2013},
   number={1},
   pages={201--206},
   issn={0022-247X},
   review={\MR{3023250}},
   doi={10.1016/j.jmaa.2013.01.018},
}

\end{biblist}\end{bibdiv}
\end{document}